\documentclass{article}

\usepackage{fullpage}
\usepackage[utf8]{inputenc}
\usepackage{amsmath,amssymb,amsthm}
\usepackage{dsfont}
\usepackage{bm}
\usepackage{soul,color}
\usepackage{graphicx}
\usepackage{enumitem}
\usepackage{ mathrsfs }
\usepackage{mathtools}
\usepackage{yhmath}
\graphicspath{{./figures/}}
\theoremstyle{plain}
\newtheorem{thm}{Theorem}
\newtheorem{lemma}{Lemma}
\newtheorem{prop}{Proposition}
\newtheorem{cor}{Corollary}

\theoremstyle{definition}
\newtheorem{fact}{Fact}
\newtheorem{defn}{Definition}

\newtheorem{assumption}{Assumption}
\newtheorem{property}{Property}
\theoremstyle{remark}
\newtheorem{myexample}{Example}
\newtheorem{oss}{Remark}

\newcommand{\virgolette}[1]{``#1''}
\newcommand{\B}{\mathbb{B}}
\newcommand{\R}{\mathbb{R}}
\newcommand{\N}{\mathbb{N}}

\newcommand{\cX}{\mathcal{X}}
\newcommand{\cY}{\mathcal{Y}}
\newcommand{\cPD}{\mathcal{PD}}
\newcommand{\cC}{\mathcal{C}}
\newcommand{\cK}{\mathcal{K}}
\newcommand{\cN}{\mathcal{N}}
\newcommand{\cO}{\mathcal{O}}
\newcommand{\cI}{\mathcal{I}}
\newcommand{\cJ}{\mathcal{J}}

\newcommand{\cS}{\mathcal{S}}
\newcommand{\cU}{\mathcal{U}}
\newcommand{\cV}{\mathcal{V}}
\newcommand{\cP}{\mathcal{P}}
\newcommand{\wt}{\widetilde}
\newcommand{\sw}{\textnormal{sw}}
\newcommand{\dom}{\mathop{\rm dom}\nolimits}

\newcommand{\Ri}{\mathbb{R}^{n_1}}
\newcommand{\Rii}{\mathbb{R}^{n_2}}

\DeclareMathOperator{\co}{co}
\DeclareMathOperator{\bd}{bd}
\DeclareMathOperator{\inn}{int}

\DeclareMathOperator*{\esssup}{ess\,sup}

\newcommand{\Sym}{\text{Sym}}

\DeclarePairedDelimiterX{\inp}[2]{\langle}{\rangle}{#1, #2}

\begin{document}

\title{Nonpathological ISS-Lyapunov Functions \\ for Interconnected Differential Inclusions}

\author{Matteo~Della Rossa \and Aneel Tanwani \and Luca Zaccarian
\thanks{M.~Della Rossa is affiliated with UCLouvain, (Louvain-La-Neuve, Belgium). E-mail: {\tt matteo.dellarossa@uclouvain.be}. A.~Tanwani and L.~Zaccarian are affiliated with LAAS -- CNRS, University of Toulouse, France.  \
L.~Zaccarian is also with Dept. of Industrial Engineering, University of Trento, Italy.
This work was partially supported by the ANR project {\sc ConVan} with grant number ANR-17-CE40-0019-01.}
}

\maketitle

\begin{abstract}
This article concerns robustness analysis for interconnections of two dynamical systems (described by upper semicontinuous differential inclusions) using a generalized notion of derivatives associated with locally Lipschitz Lyapunov functions obtained from a finite family of differentiable functions. We first provide sufficient conditions for input-to-state stability (ISS) for differential inclusions, using a class of non-smooth (but locally Lipschitz) candidate Lyapunov functions and the concept of Lie generalized derivative. In general our conditions are less conservative than the more common Clarke derivative based conditions. We apply our result to state-dependent switched systems, and to the interconnection of two differential inclusions. As an example, we propose an observer-based controller for certain nonlinear two-mode state-dependent switched systems.
\end{abstract}


%

\section{Introduction}
For analyzing stability or performance of integrated or large-scale dynamical systems, it is natural to consider them as a collection of several subsystems of lower dimension/complexity. After a certain abstraction, the behavior of the overall system can be obtained either by switching among the constituent subsystems, or through certain interconnections of the underlying subsystems, or through a combination of these. This viewpoint of analyzing complex systems provides the motivation to consider stability and robustness analysis for interconnections of switched systems.

Given a family of vector fields $\{f_1,\dots,f_K\}\subset \cC^1(\R^n\times \R^m,\R^n)$ and a switching signal $\sigma :\R^n\to \{1,\dots,K\}$, we consider the system
\begin{equation}\label{eq:IntroSwitchSys}
\dot x=f_{\sigma(x)}(x,u).
\end{equation}
For studying generalized solutions of such discontinuous systems, we extend the map $f_{\sigma(\cdot)}(\cdot,u)$, considering the Filippov regularization of~\eqref{eq:IntroSwitchSys} (see \cite{filippov1988differential}). This leads to a differential inclusion of the form (see Section \ref{sec:statedep} for details)
\begin{equation}\label{eq:IntroDiffINc}
\dot x\in F(x,u),
\end{equation}
where $F$ satisfies some regularity assumptions. Our first objective is to study asymptotic stability and robustness with respect to $u$ for system \eqref{eq:IntroDiffINc}.
We then apply our results to the analysis of interconnected systems of the form
 \begin{equation}\label{eq:IntroInterconnection}
\left\{\begin{aligned}
\dot x_1 &\in F_1(x_1,x_2,u),\\
 \dot x_2 &\in F_2(x_1,x_2,u).
 \end{aligned}\right.
\end{equation}

In the theory of nonlinear control systems, the concept of input-to-state stability (ISS), introduced in \cite{sontag89}, has been widely used to study the robustness of dynamical subsystems to external disturbances. Because of its elegant characterization in terms of Lyapunov functions, ISS is now perceived as a textbook tool for analyzing the performance of nonlinear systems,~\cite{khalil2002nonlinear}. For example, the ISS notion has been useful in analyzing interconnections of two dynamical systems, either in cascade form \cite{SonTeel95}, or in feedback by using the small-gain condition \cite{JianTeel94, JiaMarWan96, Ito19}. Moving away from the framework of conventional nonlinear systems, the ISS notion has been generalized to systems with continuous and discrete dynamics. In this regard, we find sufficient conditions in terms of slow switching for ISS of time-dependent switched systems in  \cite{VuCha07}, characterization of ISS for hybrid systems with jump dynamics in \cite{cai2009characterizations}, \cite{caiteel2013}, or for a class of differential inclusions in \cite{JayLog09}. More recently, we have seen ISS results for interconnections of hybrid systems \cite{LibNes14, Sanf14}, and time-dependent switched systems \cite{YangLibe15, tanwani2018}.

By and large, most of the aforementioned results in the literature deal with \emph{smooth} Lyapunov functions. This is partially justified by the fact that the existence of a smooth Lyapunov function is not only sufficient but also necessary for asymptotic stability of the equilibrium \cite{CLARKE199869}, \cite{teel2000}, \cite{dayawansa}, and for ISS with respect to external perturbations \cite{Lin1996}. Under some assumptions, this implications holds true also in the context of hybrid systems, as proved in~\cite{goebel2012hybrid}, \cite{caiteel2013}. The recent survey on converse Lyapunov theorems \cite{Kellett2015} provides an insightful background on such developments. However, in the context of switched and hybrid systems, the ``composite'' structure  itself provides the motivation to work with multiple smooth Lyapunov functions, see for example~\cite[Chapter 3]{liberzon} and~\cite{johansson}. More specifically, when dealing with time-dependent switched systems, these multiple Lyapunov functions can still be combined to get a smooth (with respect to the state) composite Lyapunov function. Such constructions have been seen in analyzing ISS of switched system \cite{VuCha07} and certain interconnections \cite{YangLibe15, tanwani2018}. When dealing with state-dependent switched systems, the {\em patching} of the Lyapunov functions may make the resulting common Lyapunov function non-differentiable, but locally Lipschitz in most cases. This element is seen in the analysis of asymptotic stability using piecewise differentiable functions \cite{johansson, BaiGru12}, and to some extent for establishing ISS in~\cite{heemels2007input}, \cite{HeemWei08}. For trajectory-based conditions for ISS of state-dependent switched systems, see the recent paper \cite{LiuLib18}.

This paper is about developing sufficient conditions for ISS using locally Lipschitz Lyapunov functions for the class of differential inclusions in~\eqref{eq:IntroDiffINc}. The concept of set-valued derivatives for locally Lipschitz functions, introduced in~\cite{clarke3}, \cite{bacciotti99}, is crucial to properly define the notion of \emph{derivatives along the system's trajectories}. In particular, we focus on two different notions of set-valued derivatives that we call \emph{Clarke} and \emph{Lie derivatives}, each of them being a set-valued map from the state space $\R^n$ to the real numbers, see \cite{ceragioli} and \cite{cortes} for the formal definitions. For a large class of locally Lipschitz functions called \emph{non-pathological functions} (as phrased in \cite{valadier1989entrainement}), the notion of Lie derivative leads to less conservative stability conditions, see~\cite{ceragioli} and our recent papers~\cite{ DelRosNOLCOS19} and~\cite{dellarossa19}. As another example, the Lie derivative concept has been recently used in~\cite{kamalapurkar17} to identify and remove infeasible directions of a differential inclusion of the form \eqref{eq:IntroDiffINc}, and for stability analysis using an invariance principle for state-dependent switched systems~\cite{KamRosTeel19}, based on the ideas already introduced in~\cite{SanGoe07}.

The technical content of this paper starts with the use of Lie derivatives for establishing Lyapunov-based ISS results for differential inclusions~\eqref{eq:IntroDiffINc} (but with particular attention to switched systems as in~\eqref{eq:IntroSwitchSys}), while considering non-pathological candidate Lyapunov functions. 
The following original contributions are then presented.
\begin{itemize}[leftmargin=*]
	\item  We propose a novel ISS Lyapunov result for state-dependent switched systems~\eqref{eq:IntroSwitchSys} using \emph{piecewise $\cC^1$} Lyapunov functions, also providing a numerical example that illustrates its relevance. 
	\item We study ISS of the interconnection~\eqref{eq:IntroInterconnection}, using non-smooth Lyapunov functions that satisfy a mild decrease condition based on the Lie derivative, thus generalizing the existing small gain results for state-dependent switched systems based on the Clarke derivative/gradient, as~\cite{LibNes14}, \cite{HeemWei08}. 
	\item When~\eqref{eq:IntroInterconnection} is in \emph{cascade} form, we combine two non-pathological Lyapunov functions to provide ISS certificates. These state-dependent switching results differ significantly from the existing results for interconnected time-dependent switched systems~\cite{YangLibe15, tanwani2018}. 
	\item We finally illustrate the usefulness of our results by performing output feedback stabilization of a state-dependent switched system using an observer-based controller. The arising conditions are shown to become computationally tractable in the switched linear case.
\end{itemize} 

The rest of the article is organized as follows: In Section~\ref{sec:prelim} we provide the basic definitions from non-smooth analysis, with particular attention to the nonpathological class of locally Lipschitz functions, together with the main result on ISS of system~\eqref{eq:IntroDiffINc} using locally Lipschitz Lyapunov functions. In Section~\ref{sec:statedep} we apply our result to state-dependent switched systems. In Section~\ref{sec:Interconnected}, we study interconnected differential inclusions, proposing a Lie derivative-generalization of classical small-gain and cascade arguments. We study the application of our results for feedback stabilization of switched systems in Section~\ref{sec:appStab}. In the Appendix, we prove some technical results on piecewise $\cC^1$ functions used in Section~\ref{sec:statedep}.

\section{Fundamental Tools and Results}\label{sec:prelim}
\subsection{Basic notions for differential inclusions}
We introduce here the formalism of differential inclusions with inputs, and recall the basic concepts of solutions and stability of equilibrium.
Throughout this manuscript, we consider set valued maps $F: \R^n \times \R^m  \rightrightarrows \R^n$, satisfying Assumption~\ref{Assump:Main}.
\begin{assumption}\label{Assump:Main}
Considering a set valued map $F: \R^n \times \R^m  \rightrightarrows \R^n$, we suppose that:
\begin{itemize}[leftmargin=*]
\item $F$ has nonempty, compact and convex values;
\item $F$ is locally bounded (see \cite[Definition 5.14]{rockafellar});
\item For every $u\in \R^m$, $F(\cdot,u):\R^n\rightrightarrows \R^n$ is upper semi-continuous;
\item For every $x\in \R^n$, $F(x,\cdot):\R^m\rightrightarrows \R^n$ is continuous.\hfill $\triangle$
\end{itemize} 
\end{assumption} 
The interested reader is referred to \cite{rockafellar} for a thorough discussion about continuity concepts for set-valued maps. 
We suppose that $F(0,0)=\{0\}$, and consider the differential inclusion
\begin{equation}\label{eq:diffincinp}
\dot x \in {F}(x,u),
\end{equation}
where input $u:\R_+ \to \R^m$ belongs to the set of measurable and locally essentially bounded functions, i.e. 
\[
\cU:=\left\{u:\R_+\to \R^m\,\Big\vert\begin{aligned}\, &\;\;\;\;\;\;u \text{ measurable,}\,\\ &\esssup_{0\leq \tau\leq T }\;|u(\tau)|< \infty, \;\forall T>0\end{aligned} \right \}.
\]
For the unperturbed differential inclusion 
$
\dot x \in F(x,0),
$
the hypotheses that $F(\cdot,0):\R^n\rightrightarrows \R^n$ has closed, convex and non-empty values together with upper semicontinuity are sufficient for the existence of solutions, and are sometimes referred as \emph{basic assumptions} in the literature, (see, for example, \cite{goebel2012hybrid}). On the other hand, the hypothesis that $F:\R^n\times \R^m\rightrightarrows \R^n$ is continuous in the second argument is introduced to handle a large class of inputs like $\cU$.

We introduce here the concepts of solutions: 
Given a vector $x_0\in \R^n$ and an input $u\in \cU$, $x:[0,T)\to \R^n$ (for some $T>0$ and possibly $T=\infty$) is a \emph{(Carathéodory) solution} of system \eqref{eq:diffincinp} starting at $x_0$ if $x:[0,T)\to \R^n$ is locally absolutely continuous, $x(0)=x_0$,
and $\dot x(t)\in F(x(t),u(t))$, for almost every $t\in [0,T)$.
Under the stated assumptions on the map $F:\R^n\times \R^m\rightrightarrows \R^n$, we may prove the following existence result.
\begin{prop}[Local existence]\label{prop:Existence}
Let $F: \R^n \times \R^m  \rightrightarrows \R^n$ satisfy Assumption~\ref{Assump:Main}.
Given any input $u\in \cU$, system~\eqref{eq:diffincinp} has solutions from any initial point $x_0\in \R^n$, i.e. there exists (at least) a Carathéodory solution $x:[0,T)\to \R^n$ of system~\eqref{eq:diffincinp}, for some $T>0$, with $x(0)=x_0$.
\end{prop}
\begin{proof}
Considering any input $u\in \cU$, we define $F_u:\R_+\times \R^n\rightrightarrows \R^n$ by $F_u(t,x):=F(x, u(t))$ and we prove the existence of solutions of the non-autonomous differential inclusion  
\begin{equation*}
\dot x(t)\in F_u(t,x(t)).
\end{equation*}
By hypothesis, $F_u$ is upper semicontinuous with respect to the $x$-argument.
By continuity of $F$ with respect to the second argument, for every $x\in \R^n$, we can extract a continuous function
$f(x,\cdot):\R^m\to \R^n$ such that  $f(x,u)\in F(x,u)$, for every $u\in \R^m$, see for example \cite[Example 5.57]{rockafellar} or \cite[Lemma 2.1]{Dei92}.
Thus, $f(x,u(\cdot)):\R_+\to \R^n$ is a measurable function such that $f(x,u(t))\in F(x,u(t))=F_u(t,x)$.
Since $F:\R^n\times \R^m\rightrightarrows \R^n$ is locally bounded, $F_u:\R_+\times \R^n\rightrightarrows \R^n$ is locally essentially bounded, and hence it is locally bounded by integrable functions. 
We can then apply~\cite[Corollary 5.2]{Dei92} to conclude local existence of solutions.
\end{proof}
Next, we recall the input-to-state stability (ISS) concept, firstly introduced in~\cite{sontag89}.
\begin{defn}\label{def:ISS}
System \eqref{eq:diffincinp} is \emph{input-to-state stable (ISS)} with respect to $u$  if there exist a class $\mathcal{KL}$ function $\beta$, and a class $\mathcal{K}$ function\footnote{A function $\alpha:\R_{\geq 0}\to \R$ is \emph{positive definite} ($\alpha\in\cPD$) if it is continuous, $\alpha(0)=0$, and $\alpha(s)>0$ if $ s\neq 0$. A function $\alpha:\R_{\geq 0}\to \R_{\geq 0}$ is of \emph{class $\cK$} ($\alpha \in \cK$) if it is continuous, $\alpha(0)=0$, and strictly increasing; it is of \emph{class $\cK_\infty$} if, in addition, it is unbounded. A continuous function $\beta:\R_+\times \R_+\to \R_+$ is of \emph{class $\mathcal{KL}$} if $\beta(\cdot,s)$ is of class $\cK$ for all $s$, and $\beta(r,\cdot)$ is decreasing and $\beta(r,s)\to 0$ as $s\to\infty$, for all $r$. } $\chi$ such that, for any $x_0 \in \R^n$ and for any input $u \in \cU$,
all the solutions starting at $x_0$ satisfy 
\begin{equation}\label{eq:issbounds}
|x(t)| \leq \beta(|x_0|, t)+ \chi \bigl(\esssup_{0 \leq \tau \leq t}|u(\tau)|\bigr),\;\;\;\forall t\geq 0.
\end{equation}
\end{defn}
Recalling the definition of $\cU$, bound \eqref{eq:issbounds} ensures that the solutions $x$ are uniformly bounded, and thus \emph{complete}, i.e.  $\dom(x(\cdot))=[0,+\infty)$.
It is  clear that ISS of \eqref{eq:diffincinp} implies global asymptotic stability (GAS) in the unperturbed case $u\equiv 0$.\hfill $\triangle$

\subsection{Generalized derivatives}
Our aim is to prove ISS of system~\eqref{eq:diffincinp} via non-smooth Lyapunov functions, and thus in the following we collect various notions of generalized derivatives and gradients. 
Given a locally Lipschitz function $V:\R^n \to \R$ we have the following characterization of the Clarke generalized gradient~\cite[Theorem 2.5.1, page 63]{clarke3} which is taken here as a definition.
Let $V:\R^n \to \R$ be a locally Lipschitz function, Clarke generalized gradient of $V$ at $x$ is
\begin{equation}\label{eq:limClark}
\partial V(x):= \co \left\{\lim_{k \to \infty} \nabla V(x_k) \, \vert \, x_k \to x, \;x_k \notin \cN_V \right\},
\end{equation}
 where $\cN_V \subseteq \R^n$ is the set where $\nabla V$ is not defined, which has zero Lebesgue measure by Rademacher's Theorem, and $\co(S)$ denotes the convex hull of a set $S\subseteq \R^n$.
We now introduce two different notions of generalized directional derivatives for locally Lipschitz functions with respect to  differential inclusion \eqref{eq:diffincinp}, which appeared firstly in \cite{bacciotti99}.

\begin{defn}[Set-valued directional derivatives \cite{bacciotti99, cortes}]\label{def:geneder}
Consider the unperturbed differential inclusion \eqref{eq:diffincinp} with $u\equiv 0$; given a locally Lipschitz continuous function $V:\R^n \to \R$, the \emph{Clarke generalized derivative} of $V$ with respect to $F$, denoted $\dot V_F(x)$, is defined as
\begin{equation*}
\dot V_F(x) := \{\inp{p}{f} \; | \; p \in \partial V(x),\, f \in F(x,0) \, \}.
\end{equation*}
Additionally, we define the \textit{Lie generalized derivative} of $V$ with respect to $F$, denoted $\dot{\overline{V}}_F$, as
\begin{equation*}
\dot{\overline{V}}_F(x):=\{ a \in \R \,|\, \exists f \in F(x,0): \inp{p}{f}=a, \, \forall p \in \partial V(x) \}.
\end{equation*}
These concepts can be extended to the case of a perturbed differential inclusion with input \eqref{eq:diffincinp} as follows:
\begin{equation}\label{eq:LieWithInput}
\begin{aligned}
\dot V_{F}(x,u) &:= \{\inp{p}{f} \; | \; p \in \partial V(x),\, f \in {F}(x,u) \, \},\\
\dot{\overline{V}}_{{F}}(x,u)&:=\{ a \in \R \,|\, \exists f \in {F}(x,u): \inp{p}{f}=a, \, \forall p \in \partial V(x) \}.
\end{aligned}
\end{equation}
\end{defn}
For each $(x,u)\in\R^n\times \R^m$ the sets $\dot V_{F}(x,u)$ and $\dot{\overline{V}}_{{F}}(x,u)$ are closed and bounded intervals, with $\dot{\overline{V}}_{{F}}(x,u)$ possibly empty, see~\cite{ceragioli}. In particular
\begin{equation}\label{eq:setdiffinclusion}
\dot{\overline{V}}_F(x,u) \subseteq \dot V_F(x,u).
\end{equation}
Moreover, if $V$ is continuously differentiable at $x$, one has $\partial V(x)=\{\nabla V(x)\}$ and thus  
\[
\dot{\overline{V}}_F(x,u)=\dot V_F(x,u)=\{\inp{\nabla V(x)}{f} \, \vert \, f \in F(x,u) \}.
\]
\subsection{nonpathological Functions}
We now introduce a class of locally Lipschitz functions (and not necessarily $\mathcal{C}^1$), firstly introduced in \cite{valadier1989entrainement}.

\begin{defn}\cite{valadier1989entrainement}\label{def:nonpat}
A locally Lipschitz function $V:\R^n \to \R$ is said to be  $\emph{nonpathological}$  if, given any absolutely continuous function $\varphi \in AC(\R_+,\R^n)$,  we have that  for almost every $t\in \R_+$ there exists $a_t\in \R$ such that
\[
\inp{v}{\dot\varphi(t)}=a_t,\;\;\text{for all } v\in \partial V(\varphi(t)).
\]
In other words, $\partial V(\varphi(t))$ is a subset of an affine subspace orthogonal to $\dot \varphi(t)$, for almost every $t\in \R_+$.\hfill $\triangle$
\end{defn}

The usefulness of nonpathological functions is mainly given by the following result.
\begin{prop}\cite{valadier1989entrainement}\label{prop:nonpat2}
If $V:\R^n\to \R$ is nonpathological and $\varphi :\R_+\to \R^n$ is an absolutely continuous function, then the set
\[
\{\inp{p}{\dot \varphi(t)}\;\vert\; p\in \partial V(\varphi(t))\},
\]
is equal to the singleton $\{\frac{d}{dt}V(\varphi(t))\}$ for almost every $t\in \R_+$.
\end{prop}
\begin{oss}\label{rmk:nonpathologicalprop}
Given a nonpathological function $V:\R^n \to \R$, for any $u\in \cU$, any initial condition $x_0\in \R^n$ and any solution $x:\dom(x(\cdot))\to\R^n$ of \eqref{eq:diffincinp}, we have that
\begin{equation}\label{eq:maininclusion}
\frac{d}{dt}V(x(t))\in \dot{\overline{V}}_{{F}}(x(t),u(t))
\end{equation}
for almost every $t\in \dom(x(\cdot))$. In fact, by Proposition \ref{prop:nonpat2}, for almost every $t\in \dom(x(\cdot))$, we have that 
\[
\begin{aligned}
 \big \{ &\frac{d}{dt}V(x(t)) \big \}=\{\inp{p}{\dot x(t)}\;\vert\;p\in\partial V(x(t))\}\\&\subseteq  \{ a \in \R \,|\, \exists f \in {F}(x(t),u(t)): \inp{p}{f}=a, \forall p \in \partial V(x(t)) \}\\&=\dot{\overline{V}}_{{F}}(x(t),u(t)).
\end{aligned}
\]
\end{oss}
nonpathological functions form a large class of functions which clearly includes $\mathcal{C}^1(\R^n,\R)$, we recall here some important properties of this family of functions, for the proofs we refer to~\cite{valadier1989entrainement} and~\cite{BacCer03}.
\begin{lemma}\label{lemma: PropNonPatFunc}
 The set of nonpathological functions is closed under addition, multiplication by scalars and pointwise maximum operator. More precisely, if $V_1,V_2:\R^n\to \R$ are nonpathological then $V_1+V_2$, $\max\{V_1(x),V_2(x)\}$ and $\lambda V$ ($\lambda\in \R$) are nonpathological.
Moreover, if $V:\R^n \to \R$ is locally Lipschitz continuous and has \emph{at least} one of the following properties:
\begin{itemize}
\item Continuously differentiable, 
\item Clarke-regular  (see \cite[Definition 2.3.4.]{clarke3}), 
\item convex/concave, 
\item semiconvex/semiconcave,
\end{itemize}
then $V$ is nonpathological.
\end{lemma}
 In our paper~\cite{dellarossa19} the class of nonpathological functions obtained from pointwise maximum and minimum combinations over a finite set of continuously differentiable functions has been studied, by proposing sufficient conditions for asymptotic stability of state-dependent switching systems.
 In particular, explicitly exploiting the properties of the max-min structure, we have shown in~\cite{dellarossa19} the advantages of using the Lie derivative concept (Definition~\ref{def:geneder}), providing several examples where the Clarke derivative approach is too conservative. In this work instead, we study ISS of differential inclusions with inputs, considering the broader class of nonpathological candidate Lyapunov functions, and investigating how the nonpathological property could be exploited in the context of interconnected differential inclusions.
In Section \ref{sec:statedep}  we will define a family of locally Lipschitz functions and we will prove that it is a subset of the nonpathological functions. A similar class is considered also in our parallel work~\cite{dellarossa20}, in the context of hybrid systems composed by \emph{continuous} differential inclusions over restricted domains.

\subsection{ISS-Lyapunov Result}
 We now provide sufficient conditions for ISS of system \eqref{eq:diffincinp}. In what follows, due to the fact that the set $\dot{\overline{V}}_{{F}}(x,u)$ is possibly empty, we adopt the convention $\max \emptyset =-\infty$.
 The novelty of the following ISS-Lyapunov result lies in the fact that we require \emph{Lie generalized derivative} of the Lyapunov function to be negative definite. Recalling the inclusion~\eqref{eq:maininclusion}, this statement can be seen as a generalization of the existing results on ISS of differential inclusions relying on the notion of Clarke derivative, in particular~\cite{LibNes14}.
\begin{thm}\label{teo:mainres}
Let $V :\R^n \to \R$ be a locally Lipschitz and nonpathological function such that there exist $\underline{\alpha}, \overline{\alpha} \in \mathcal{K}_{\infty}$, $\rho\in\cPD$ and
$\gamma \in \mathcal{K}$  such that
\begin{align}
&\underline{\alpha}(|x|) \leq V(x) \leq \overline{\alpha}(|x|), \label{eq:cond1}\\
& V(x)> \gamma(|u|)\;\;\Rightarrow\;\;\max \dot{\overline{V}}_{F}(x,u) \leq -\rho(|x|), \label{eq:cond2}
\end{align}
then system \eqref{eq:diffincinp} is ISS w.r.t. $u$, and  $V$ is called a \emph{nonpathological ISS-Lyapunov function} for  system \eqref{eq:diffincinp}.
\end{thm}
\begin{proof}
Let us consider any initial condition $x_0\in \R^n$ and any input $u\in \cU$. Let $x:\dom(x(\cdot)) \to \R^n$ be a solution of system \eqref{eq:diffincinp} starting at $x_0$ and with input $u$. 
Let us note that the function $V \circ x : \dom(x(\cdot)) \to \R$ is absolutely continuous because it is the composition of a locally Lipschitz continuous function and an absolutely continuous function. Then $\frac{d}{dt}V(x(t))$ exists almost everywhere.
Since $V$ is nonpathological, recalling \eqref{eq:maininclusion} in Remark~\ref{rmk:nonpathologicalprop}, we have that $\frac{d}{dt}V(x(t))\in\dot{\overline{V}}_{{F}}(x(t), u(t))$ for almost every $t \in \dom(x(\cdot))$.
From equation~\eqref{eq:cond2},  we have that
\begin{equation*}
\frac{d}{dt}V(x(t)) \leq -\rho(|x(t)|),
\end{equation*}
for almost every $t\in\dom(x(\cdot))$ such that $V(x(t))>\gamma(|u(t)|)$.
The proof is completed by  following standard approaches as those in~\cite{Sontag1995},~\cite[Theorem 4.18]{khalil2002nonlinear}. For the interested reader, the argument is fully developed in~\cite[Proof of Theorem 5.4]{DellaROPhd20}.
\end{proof}
\begin{oss}\label{rmk:equiv}
As already observed in the literature, for example in \cite{cai2009characterizations}, under some assumptions the existence of $\rho\in \cPD$ and $\gamma \in \mathcal{K}$ such that condition \eqref{eq:cond2} holds is equivalent to the existence of two functions $\widehat \rho\in \mathcal{K}_{\infty}$ and $\widehat \gamma\in \cK$ such that 
\begin{equation}\label{eq:cond2eqival}
\max \dot{\overline{V}}_{{F}}(x,u) \leq -\widehat \rho(|x|)+\widehat\gamma(|u|) \;\; \forall (x,u)\in \R^n \times \R^m.
\end{equation}
Indeed, implication \eqref{eq:cond2eqival} $\Rightarrow$ \eqref{eq:cond2} holds by choosing $\rho:=\frac{1}{2}\widehat \rho$ and $\gamma=\widehat \rho^{-1}\circ 2\widehat \gamma$. 
The converse implication \eqref{eq:cond2} $\Rightarrow$ \eqref{eq:cond2eqival}  holds if $F(0,0)=\{0\}$ and $\rho\in \cK_\infty$. Indeed, it suffices to choose $\widehat \rho:=\rho$ and $\widehat \gamma\in \cK$ such that $\widehat \gamma(r) \geq \max\{0, \widehat \gamma_0(r)\}$, for all $r\in \R_+$, where
\[
\widehat \gamma_0(r):=\max\left\{\max \dot{\overline{V}}_{{F}}(x,u)+\rho(|x|) \,\vert\,|x|\leq\gamma(r),\;|u| \leq r \right\}.
\]
It is possible to show that $\widehat \gamma_0$ above is well defined and $\widehat \gamma_0\in\cK$, using local boundedness of $\partial V$ and regularity of $F:\R^n\times \R^m\rightrightarrows \R^n$.
Moreover, using~\eqref{eq:cond1}, another equivalent formulation of condition \eqref{eq:cond2eqival} corresponds to asking that there exist two functions $\wt \rho \in \mathcal{K}_{\infty}$ and $\wt \gamma \in \cK$ such that 
\begin{equation}\label{eq:cond2equival2}
\max \dot{\overline{V}}_{{F}}(x,u) \leq -\wt \rho(V(x))+\wt\gamma(|u|), \;\; \forall (x,u)\in \R^n \times \R^m.
\end{equation}
 The advantage of \eqref{eq:cond2equival2} is that in this formulation the function $\wt \rho$ represents the decay rate of $V$ along the solutions.\hfill$\triangle$
\end{oss}

\section{ISS for State Dependent Switched Systems}\label{sec:statedep}
In this section, we apply our ISS result to a specific differential inclusion with inputs arising from a suitable regularization of state-dependent switched systems.
We introduce the concept of a \virgolette{well-behaved} partition of the state space, of the associated switched system, and of a family of functions related to this partition, and finally we provide the specialization of Theorem \ref{teo:mainres} in this setting.
\begin{defn}[Proper State-Space Partition]\label{def:Ncov}
Given a finite set of indexes $\cI:=\{1,\dots K\}$, let us consider closed sets $X_1,\dots,X_K\subseteq \R^n$ and open sets $\cO_1,\dots \cO_K \subseteq \R^n$ such that
\begin{enumerate}
\item[a)] $\bigcup_{i=1}^KX_i=\R^n$,
\item[b)] $X_i\subseteq \cO_i$, for all $i\in \cI$,
\item[c)] $\overline{\inn(X_i)}=X_i$, for all $i\in \cI$,
\item[d)] For every $i\in \cI$, $\bd(X_i)$ has zero Lebesgue measure,
\item[e)] $X_i\cap X_j=\bd(X_i)\cap \bd(X_j)$, for all $i,j\in \cI$, $i\neq j$,
\end{enumerate}
In this situation, we say that $\cX:=\{X_i,\cO_i\}_{i\in \cI}$ is a \emph{proper partition} of $\R^n$. We define $\partial X:=\cup_{i\in \cI}\bd(X_i)$, and we underline that $\partial X$ has zero Lebesgue measure.\hfill$\triangle$
\end{defn}
Given a proper partition $\cX$ of $\R^n$, we can introduce an \virgolette{index indicator map}, that is a set valued map $\cI_\cX:\R^n \rightrightarrows \cI$ defined as
\begin{equation}\label{eq:funcI}
\cI_{\mathcal{X}}(x):=\{i\in \cI \;\vert\; x\in X_i\}.
\end{equation}
We underline that $\cI_{\mathcal{\cX}}$ is almost everywhere single valued. In fact, by Definition~\ref{def:Ncov}, Item e), if $x\in \inn(X_\ell)$ for some $\ell\in \cI$ then $\cI_\cX(x)=\{\ell\}$.
\begin{defn}[State-Dependent Switched System]\label{def:StateDep}
Given $\cX=\{X_i,\cO_i\}_{i\in \cI}$ a proper partition of $\R^n$, consider $f_i\in \cC^1(\cO_i\times \R^m,\R^n)$, $i\in\cI$. A \emph{state-dependent switching signal} associated to $\cX$ is a function $\sigma:\R^n\to \cI$ such that
\begin{equation}\label{eq:switchingsignal}
\sigma(x)= i, \;\;\;\text{ if }x\in \inn(X_i),\\
\end{equation}
and the (perturbed) \emph{state-dependent switched system} associated to $\{X_i,\cO_i,f_i\}_{i\in \cI}$ is the differential equation
\begin{equation}\label{eq:switchingdisc}
\dot x=f_{\sigma(x)}(x,u).
\end{equation}
\end{defn}
We note here that, given a proper partition $\cX=\{X_i,\cO_i\}$, a state dependent switching signal associated to it is not uniquely defined: the value of $\sigma$ remains unspecified on the null-measure set $\partial \cX$. We now clarify why this ambiguity does not affect the solution set of the corresponding state-dependent switched system.

System \eqref{eq:switchingdisc} has a discontinuous right-hand side in the first argument thus it may not have any Carathéodory solutions at the discontinuity points of $f_{\sigma(\cdot)}(\cdot, u)$, see~\cite{cortes}. Many possible definitions of ``generalized solutions'' for discontinuous dynamical system are possible (see for example \cite{ceragioli} or \cite{cortes}); we consider the concept of \emph{Filippov solutions}, introduced firstly in \cite{filippov1988differential}. More formally, we define $\text{Fil}$, the Filippov regularization of the discontinuous map $f_{\sigma}$, as
\begin{equation*}
\begin{aligned}
\text{Fil}(f_{\sigma})(x,u)&:=\bigcap_{\delta>0}\bigcap_{\mu(S)=0}\overline{\co}\{f_{\sigma(y)}(y,u) \; \vert\; y\in \B(x,\delta)\setminus S\}\\
&=:F^\sw(x,u),
\end{aligned}
\end{equation*}
where $\mu$ denotes the Lebesgue measure.
Under the hypotheses in Definitions~\ref{def:Ncov} and~\ref{def:StateDep}, it can be proven that
\[
F^\sw(x,u)=\co\{f_i(x,u)\;\vert\;i\in \cI_\cX(x)\},
\]
see for example~\cite{filippov1988differential},~\cite{cortes} and~\cite{HeemWei08}.
Summarizing, we consider the regularized differential inclusion
\begin{equation}\label{eq:swsystem}
\dot x\in F^\sw(x,u)=\co\{f_i(x,u)\;\vert\;i\in \cI_\cX(x)\}
\end{equation}
considering again signals $u\in \cU$.
 Since it is easily verified that $f_{\sigma(x)}(x,u)\in F^\sw(x,u)$ for all $(x,u)\in \R^n\times \R^m$, we have that any solution of~\eqref{eq:switchingdisc} is also a solution of~\eqref{eq:swsystem}. Moreover, it can be seen that $F^\sw:\R^n\times \R^m\rightrightarrows \R^n$ satisfies Assumption~\ref{Assump:Main}, and thus by Proposition~\ref{prop:Existence} we have existence of solutions of~\eqref{eq:swsystem} from any initial condition and any input $u\in \cU$ (which was not the case for~\eqref{eq:switchingdisc}, see~\cite{cortes}). Then, characterizing desirable (stability/convergence) properties of the solutions of~\eqref{eq:swsystem} indirectly also characterizes the solutions of~\eqref{eq:switchingdisc}. In particular,
 a \emph{Filippov solution} of system \eqref{eq:switchingdisc} is by definition a solution of the differential inclusion \eqref{eq:swsystem}, according to the definition given in Section~\ref{sec:prelim}.

We introduce here a family of locally Lipschitz functions that we propose as candidate Lyapunov functions for~\eqref{eq:swsystem}. 
\begin{defn}[Piecewise $\cC^1$ Functions]\label{defn:piecewise}
Consider $\cJ=\{1,\dots,N\}$ and $\cY=\{Y_j,\cV_j\}_{j\in \cJ}$, a proper partition of $\R^n$. A continuous function $V:\R^n\to \R$ is called a \emph{piecewise $\cC^1$ function} with respect to the proper partition $\cY$ (and we write $V\in \mathscr{P}(\cY)$) if there exist real-valued functions $V_1,\dots V_N$ such that
\begin{enumerate}
\item $V_j\in \cC^1(\cV_j,\R)$ for each $j \in \cJ$,
\item $V(x)=V_j(x)$, if $x \in Y_j$.\hfill $\triangle$
\end{enumerate}
\end{defn}
Piecewise $\cC^1$ functions with respect to a proper partition are a particular kind of \emph{``piecewise $\cC^1$ functions''} defined in \cite{Sch12}.
\begin{prop}\label{lemma:propertiesofppf}
Consider $V\in \mathscr{P}(\cY)$, with respect to a proper partition $\cY=\{Y_j,\cV_j\}_{j\in \cJ}$, in the sense of Definition~\ref{defn:piecewise}. Then the following hold:
\begin{enumerate}[leftmargin=*] 
	\item $V$ is locally Lipschitz, nonpathological and 
\begin{equation}\label{eq:gengradPiece}
\partial V(x)=\co\left\{\nabla V_j(x)\;\vert\;j\in \cI_\cY(x)\right\}.
\end{equation}
\item Given $F^\sw:\R^n\times \R^m \rightrightarrows \R^n$ defined as in \eqref{eq:swsystem}, we have 
\begin{equation}\label{eq:PPFLieDerivative}
\dot{\overline{V}}_{F^\sw}(x,u)=\left \{ a \in \R \;\Bigl \vert
\begin{aligned} \;& \exists f \in F^\sw(x,u): \\&\inp{\nabla V_j(x)}{f}=a, \, \forall j\in \cI_\cY(x) \end{aligned} \right \}.
\end{equation}
\end{enumerate}
\end{prop}
We postpone the proof of Proposition~\ref{lemma:propertiesofppf} to the Appendix.
We can now specialize the results stated in previous sections in this setting. First, we consider a candidate Lyapunov function in the class of \emph{piecewise $\cC^1$ functions} $\cP(\cY)$, where the proper partition $\cY$ does not necessarily coincide with $\cX$, the proper partition associated with the considered switched system. Then we present specifically the case $\cX=\cY$, see the subsequent Remark~\ref{rmk:ComparisonCorollaries} for further discussion.
\begin{cor}[ISS for state dependent switching]\label{cor:CoroSw1}
Consider a proper partition $\mathcal{X}=\{X_i, \cO_i\}_{i\in \cI}$ and an associated switched system \eqref{eq:swsystem}. Let us consider another proper partition $\cY=\{Y_j,\cV_j\}_{j\in \cJ}$ and let  $V\in \mathscr{P}(\cY)$. Suppose that there exist $\underline{\alpha}, \overline{\alpha}  \in \mathcal{K}_{\infty}$, $\rho\in \cPD$ and $\gamma \in \mathcal{K}$  such that:
\begin{enumerate}[leftmargin=*,label=\textbf{\Alph*)}]
\item for each $j\in \cJ$, for each $x\in Y_j$,
\[
\underline{\alpha}(|x|) \leq V_j(x) \leq \overline{\alpha}(|x|);
\]
\item for each $j\in \cJ$, for each $x\in \inn(Y_j)\setminus \partial X$ and for each $u\in \R^m$,
\[
  V(x)> \gamma(|u|)\;\;\Rightarrow\;\;\inp{\nabla V_j(x)}{f_{\sigma(x)}(x,u)}\leq -\rho(|x|);\label{eq:LiecondswitchingBAsic1}
 \]
 \item for each $(x,u)\in \partial X \times \R^m$,
 \[ 
V(x)> \gamma(|u|)\;\;\Rightarrow\;\;\max \dot{\overline{V}}_{F^\sw}(x,u) \leq - \rho(|x|) ;\label{eq:LiecondswitchingBAsic2}
\]
\end{enumerate}
then system \eqref{eq:swsystem} is ISS.
\end{cor}
\begin{proof}
By the nonpathological property  of $V$ established in Proposition~\ref{lemma:propertiesofppf}, we can apply Theorem \ref{teo:mainres}. Getting inequality~\eqref{eq:cond1} from \textbf{A)} is straightforward. We check inequality~\eqref{eq:cond2} decomposing $\R^n$  as follows:
\[
\R^n=\Big (\bigcup_{j\in \cJ}\inn(Y_j)\setminus\partial X \Big )\cup\partial X\cup \big (\partial Y\setminus\partial X \big ).
\]
Consider first a point  $x\in\inn(Y_j)\setminus \partial X$ for some $j\in \cJ$. Function $V\in \mathscr{P}(\cY)$ is $\cC^1$ at $x$, and $F^\sw(x,u)=\{f_{\sigma(x)}(x,u)\}$ is single-valued. Thus, for any $u\in \R^m$, we have 
\[
\dot{\overline{V}}_{F^\sw}(x,u)=\{\inp{\nabla V_j(x)}{f_{\sigma(x)}(x,u)}\},
\]
and by \ref{eq:LiecondswitchingBAsic1}, the implication in~\eqref{eq:cond2} holds.\\
If $x\in \partial X$, the assertion follows directly by  \ref{eq:LiecondswitchingBAsic2}.\\
As the last step, consider a point $x\in \partial Y\setminus \partial X$, and thus $x\in \inn(X_i)$ for some $i\in \cI$. In particular at $x$, we have $F^\sw(x,u)=\{f_i(x,u)\}$. Consider $u\in \R^m$ and suppose that $V(x)>\gamma(|u|)$. Recalling~\eqref{eq:gengradPiece} in Lemma~\ref{lemma:propertiesofppf} and by Definition~\ref{def:Ncov}, for each $j\in \cI_\cY(x)$, there exists a sequence $x_k^j\to x$ such that $x_k^j\in \inn(Y_j)\cap \inn(X_i)$ for all $k\in \N$. By continuity of $V$ and $\gamma$, we can suppose $V(x^j_k)>\gamma(|u|)$. At these points, from~\ref{eq:LiecondswitchingBAsic1}, we get $\inp{\nabla V_j(x^j_k)}{f_{i}(x^j_k,u)}\leq -\rho(|x_j^k|)$. By continuity of $\nabla V_j$, $f_i$ and $\rho$ we have 
\[
\inp{\nabla V_j(x)}{f_{i}(x,u)}\leq -\rho(|x|),
\]
for each $j\in \cI_\cY(x)$.\\ We have proved~\eqref{eq:cond2} for all $x\in \R^n$, concluding the proof.
\end{proof}
We underline that we are not explicitly checking~\eqref{eq:cond2} on the zero Lebesgue measure set $\partial Y\setminus \partial X$: this is possible thanks to the continuity of $F^\sw$ when restricted to $\inn(X_i)$ for some $i\in \cI$.
As a special case of Corollary~\ref{cor:CoroSw1}, we present the situation $\cX=\cY$.

\begin{cor}\label{cor:maincor}
Consider a proper partition $\mathcal{X}=\{X_i, \cO_i\}_{i\in \cI}$ and the associated switched system \eqref{eq:swsystem}. Consider $V\in \mathscr{P}(\cX)$ such that there exist $\underline{\alpha},\overline{\alpha} \in \mathcal{K}_{\infty}$, $\rho\in \cPD$ and $\gamma \in \mathcal{K}$  satisfying
\begin{enumerate}[leftmargin=*,label=\textbf{\Alph*)'}]
\item for each $i \in \cI$, for each $x\in X_i$\label{item:Cor2Cond1}
\[
\underline{\alpha}(|x|) \leq V_i(x) \leq \overline{\alpha}(|x|);
\]
\item for each $i\in \cI$, for each $(x,u)\in  \inn(X_i) \times \R^m$,\label{item:Cor2Cond2}
\[
 V(x)> \gamma(|u|)\;\;\Rightarrow\;\;\inp{\nabla V_i(x)}{f_i(x,u)}\leq -\rho(|x|);
 \]
 \item for each $(x,u)\in \partial X \times \R^m$, \label{item:Cor2Cond3}
 \[ 
V(x)> \gamma(|u|)\;\;\Rightarrow\;\;\max \dot{\overline{V}}_{F^\sw}(x,u) \leq - \rho(|x|);
\]\label{eq:Liecondswitching}
\end{enumerate}
\vspace{-0.5cm}
then~system \eqref{eq:swsystem} is ISS.
\end{cor}
\begin{oss}[Comparison between Cor.~\ref{cor:CoroSw1} and Cor.~\ref{cor:maincor}]\label{rmk:ComparisonCorollaries}
To check the conditions of Corollary \ref{cor:maincor} for each $i\in \cI$, we must find a smooth Lyapunov function $V_i$ for the system $f_i$ on the set $\cO_i\supset X_i$. Then we must construct a continuous function $V$ by \virgolette{gluing together} the $V_i$'s on $\partial X$ and finally check the Lie-based condition~\ref{eq:Liecondswitching} on the switching surface $\partial X$. On the other hand, in some situations, it can be difficult to construct a smooth Lyapunov function even for a single subsystem in its region of activation if the subsystem is unstable. For this reason, in Corollary~\ref{cor:CoroSw1} we allow the candidate Lyapunov functions to be possibly non-differentiable in the interior of the $X_i$, and we do not need to check the conditions at the point of non-differentiability of $V$, since  $F(\cdot, u)$ is continuous in a neighborhood of $\partial Y\setminus \partial X$.\hfill $\triangle$
\end{oss}
In the following example, mainly inspired by~\cite[Example 1]{johansson}, we apply the result presented in~Corollary~\ref{cor:maincor} to a particular state-dependent switched system.
\begin{myexample}\label{ex:flower}
Consider the proper state-space partition $\mathcal{X}=\{X_i, \R^2\}_{i\in \{1,2\}}$ of $\R^2$ defined by
\[
X_i:=\{x \in \R^2 \;\vert\; x^\top Q_i x\geq 0\},
\]
with $Q_1= \begin{bsmallmatrix} -1 & 0\\0&1  \end{bsmallmatrix}$, and $Q_2=-Q_1$. Define the switching system
\[
\dot x =\begin{cases}
A_1x+Bu, & \text{if}\;x  \in X_1,\\
A_2 x+Bu, & \text{if}\;x  \in X_2,
\end{cases}
\]
with $A_1=  \begin{bsmallmatrix} -\varepsilon & a_1 \\ -a_2 & -\varepsilon \end{bsmallmatrix} $, $A_2=  \begin{bsmallmatrix} -\varepsilon & a_2 \\ -a_1 & -\varepsilon \end{bsmallmatrix}$ with $0<\varepsilon<a_1\leq a_2$ and $B\in \R^{2\times m}$ arbitrary.  We want to study the resulting Filippov regularization as defined in~\eqref{eq:swsystem}; for a graphical representation of the unperturbed case $u\equiv 0$, for some specific selections of $a_1,a_2,\varepsilon$, see~\cite[Figure 1]{DelRosNOLCOS19}. Consider the function $V:\R^2\to \R$ defined by
\[
V(x)=\begin{cases}
V_1(x)=x^\top P_1 x \;\;\;\;\;\text{if}\; x\in X_1,\\
V_2(x)=x^\top P_2 x\;\;\;\;\;\text{if}\; x \in X_2,
\end{cases}
\]
with $P_1=\begin{bsmallmatrix}
a_2 & 0\\ 0 & a_1 
\end{bsmallmatrix}$, $P_2=\begin{bsmallmatrix}
a_1 & 0\\0 & a_2 
\end{bsmallmatrix}$.
We now show that $V$ is an ISS Lyapunov function, in the sense of Corollary~\ref{cor:maincor}.
Since $P_2-P_1=(a_2-a_1)Q_2$, $V$ is continuous (and thus $V\in \mathscr{P}(\cX)$) and moreover it satisfies item~\ref{item:Cor2Cond1} of Corollary~\ref{cor:maincor} with $\underline{\alpha}(s)=a_1s^2$ and $\overline{\alpha}(s)=a_2 s^2$. Moreover, it can be seen that
\[
A_i^\top P_i+P_iA_i+\varepsilon I\preceq0, \;\;\;\;\text{for any }i\in \{1,2\};
\]
thus, following the reasoning of \cite[Lemma 4.6]{khalil2002nonlinear}, it can be proven that, for any $0<\varepsilon'<\varepsilon$ we have
\begin{equation}\label{eq:EXFirstSTep}
|x|\geq b|u|\,\,\Rightarrow x^\top (A_i^\top P_i+P_iA_i)x+2x^\top P_iBu\leq -\varepsilon'|x|^2, 
\end{equation}
for some $b>0$ depending on $A_1,A_2,B,P_1,P_2$ and $\varepsilon'$.\\
It is easy to check $\partial X=\{x\in \R^2\;\vert\;x^\top Q_1 x=0\}=\mathcal{R}_1\cup \mathcal{R}_2$ where $\mathcal{R}_1=\{[x_1\;x_2]^\top \in \R^2\;\vert\;x_1=x_2\}$ and $\mathcal{R}_2=\{[x_1\;x_2]^\top \in \R^2\;\vert\;x_1=-x_2\}$. Consider $z\in \partial X$, and firstly, let us suppose that $z\in \mathcal{R}_1$, and define $\mu:=\frac{|z|}{\sqrt{2}}$; we can write $z=\mu[1\;1]^\top$ (w.l.o.g. supposing that $z$ lies in the first quadrant). By Proposition~\ref{lemma:propertiesofppf} and~\eqref{eq:swsystem}, we have that 
\[
\partial V(z)=2\mu\co\left\{v_1,v_2 \right\} \;\;\text{and }\;F(z,u)=\mu\co\{f_1,f_2\}+Bu,
\]
where $v_1=[a_2\;a_1]^\top$, $v_2=[a_1\;a_2]^\top$ and $f_i=A_i[1\;1]^\top$ for $i\in\{1,2\}$. Recalling Definition~\ref{def:geneder}, we have that $\dot{\overline{V}}_{F^\sw}(z,u)\neq \emptyset$
 if and only if there exists $\lambda\in \R$, $0\leq \lambda\leq 1$, such that
 \[
[\mu(\lambda f_1+(1-\lambda) f_2)+Bu]^\top v_1=[\mu(\lambda f_1+(1-\lambda)f_2)+Bu]^\top v_2.
 \]
 Simplifying, we obtain 
$
(a_1+a_2) \mu=[-1\;\,1]^\top Bu$ (note that the dependence on $\lambda$ cancels out). Now, it is clear that if $(a_1+a_2)\mu\geq \sqrt2\|B\||u|$, this equation has no solution\footnote{Here, and in what follows, we consider the 2-induced matrix norm, i.e., given $M\in \R^{n\times m}$ we define
\[
\|M\|:=\sup_{|x|=1}|Mx|
\]}. Recalling the definition of $\mu$ we have proved that
\begin{equation}\label{eq:ExConD2}
z\in \mathcal{R}_1 \text{ and } |z|\geq \frac{2}{a_1+a_2}\|B\||u|\;\Rightarrow \; \dot{\overline{V}}_{F^\sw}(z,u)= \emptyset.
\end{equation}
 Since $A_i=R^\top A_jR$  and $P_i=R^\top P_jR$  with $R=\begin{bsmallmatrix}0 &1\\   -1&0\end{bsmallmatrix}$ for any $i,j\in\{1,2\}$, $i\neq j$, it can be seen that \eqref{eq:ExConD2} holds also for $z\in\mathcal{R}_2$.
Fixing $\varepsilon'<\varepsilon$ in~\eqref{eq:EXFirstSTep} and defining  $\gamma(s):=a_2\max\{bs, \tfrac{2\|B\|}{a_1+a_2}\,s\}$ and $\rho (s):=\varepsilon' s^2$, we have proved Items~\ref{item:Cor2Cond2} and~\ref{item:Cor2Cond3} of Corollary~\ref{cor:maincor}, establishing that $V$ is a (nonpathological) ISS Lyapunov function, and thus the system is ISS. \hfill $\triangle$
 \end{myexample}
\begin{oss}
The construction proposed in Example~\ref{ex:flower} can be generalized to broader settings. The main idea is the following:  consider a proper partition $\mathcal{X}=\{X_i, \cO_i\}_{i\in \cI}$ and the switched system \eqref{eq:swsystem} with $f_i(x,u)=A_ix+g_i(u)$, with $A_i\in \R^{n\times n}$ and $g_i:\R^m\to \R^n$ satisfying a linear growth condition, i.e. $\exists\;L>0$ such that $g_i(u)\leq L|u|$ for all $u\in \R^m$, for all $i\in \cI$. Suppose that a piecewise quadratic $V\in \cP(\cX)$  satisfies conditions~\ref{item:Cor2Cond1} and~\ref{item:Cor2Cond2} of Corollary~\ref{cor:maincor}. Then, if for the \emph{unperturbed} system the Lie derivative on the discontinuity surface $\partial X$ is empty, i.e.
\[
\dot{\overline{V}}_{F^\sw}(x,0)=\emptyset,\;\;\;\forall x\in \partial X,
\]
it follows that condition~\ref{item:Cor2Cond3} is satisfied. Intuitively speaking, this is due to the fact that the set $\dot{\overline{V}}_{F^\sw}(x,u)$ remains empty for all $x\in \R^n$ and all $u\in \R^m$ small enough.  For further insight regarding the   construction of piecewise $\cC^1$ (or piecewise quadratic) Lyapunov functions for (unperturbed) state-dependent switching systems, we refer the reader to~\cite{johansson},~\cite{IertrennVasca},~\cite{dellarossa19} and references therein. \hfill $\triangle$
\end{oss}

\section{Interconnected Differential Inclusions}\label{sec:Interconnected}
In this section, we use Theorem~\ref{teo:mainres} to study stability of feedback and cascade interconnections of two systems modeled by differential inclusions. 
\subsection{Feedback Interconnection and Small Gain Theorem}\label{sec:smallGain}
For the system shown in Figure~\ref{fig:blockdiagram}, we establish  ISS of the interconnected system by constructing a Lyapunov function from two (nonsmooth) ISS-Lyapunov functions associated with the two  subsystems. 

 Consider $F_1:\R^{n_1}\times \R^{n_2}\times \R^m\rightrightarrows \R^{n_1}$ and $F_2:\R^{n_1}\times \R^{n_2}\times \R^m\rightrightarrows \R^{n_2}$ and suppose that they are locally bounded, have non empty, compact and convex values and are upper semicontinuous in the first two arguments and continuous in the third one. Consider the interconnection
 \begin{subequations}\label{eq:InterconnectedDiffInc}
\begin{align}
&\dot x_1\in F_1(x_1,x_2,u),\label{eq:subsys1}\\
& \dot x_2\in F_2(x_1,x_2,u).\label{eq:subsys2}
\end{align}
\end{subequations}
We introduce the notation $x=(x_1,x_2):=(x_1^\top,x_2^\top)^\top \in \R^n=\R^{n_1+n_2}$ and the augmented differential inclusion
\begin{equation}\label{eq:CompDiffInc}
\dot x\in F(x,u):=\begin{pmatrix}
F_1(x_1,x_2,u)\\
F_2(x_1,x_2,u)
\end{pmatrix}.
\end{equation}
We start our construction by assuming the existence of ISS-Lyapunov functions for the two subsystems, in order to conclude ISS of the overall interconnection~\eqref{eq:CompDiffInc}. These types of assumptions characterize classical ISS approaches~\cite{JianTeel94, JiaMarWan96}, also used in the recent works~\cite{HeemWei08}, ~\cite{LibNes14}. The novelty introduced here lies in the fact that we consider \emph{nonpathological} ISS Lyapunov functions satisfying the ``relaxed''  conditions involving the Lie derivative presented in Theorem~\ref{teo:mainres}, as formalized in the following statement.

\begin{assumption}\label{ass:InterLyapFuncSub}
Suppose that there exist nonpathological functions $V_1:\Ri\to \R$ and $V_2:\Rii\to\R$  such that
\begin{enumerate}
\item[1a)] There exist $\underline{\alpha}_1,\overline{\alpha}_1\in \cK_\infty$ satisfying
\[
\underline{\alpha}_1(|x_1|)\leq V_1(x_1)\leq \overline{\alpha}_1(|x_1|),\;\;\;\forall\;x_1\in \Ri.
\] 
\item[1b)] There exist $\underline{\alpha}_2,\overline{\alpha}_2\in \cK_\infty$ satisfying
\[
\underline{\alpha}_2(|x_2|)\leq V_2(x_2)\leq \overline{\alpha}_2(|x_2|),\;\;\;\forall\;x_2\in \Rii.
\] 
\item[2a)]There exist $\rho_1\in \cPD$, and $\chi_1,\gamma_1\in \cK$ satisfying
\begin{equation*}
\begin{aligned}
V_1(x_1)> \max\{\chi_1&(V_2(x_2)), \gamma_1(|u|)\}\\ &\Downarrow\\\max \dot{\overline{V}}_{1,F_1}(x_1,x_2,&u)\leq -\rho_1(V_1(x_1))
\end{aligned}
\end{equation*}
\item[2b)]There exist $\rho_2\in \cPD$, and $\chi_2,\gamma_2\in \cK$ satisfying
\begin{equation*}
\begin{aligned}
V_2(x_2)> \max\{\chi_2&(V_1(x_1)), \gamma_2(|u|)\}\\&\Downarrow\\\max \dot{\overline{V}}_{2,F_2}(x_1,x_2,&u)\leq -\rho_2(V_2(x_2))
\end{aligned}
\end{equation*}
\end{enumerate}
\end{assumption}

\begin{figure}
\begin{center}
\includegraphics[scale=1]{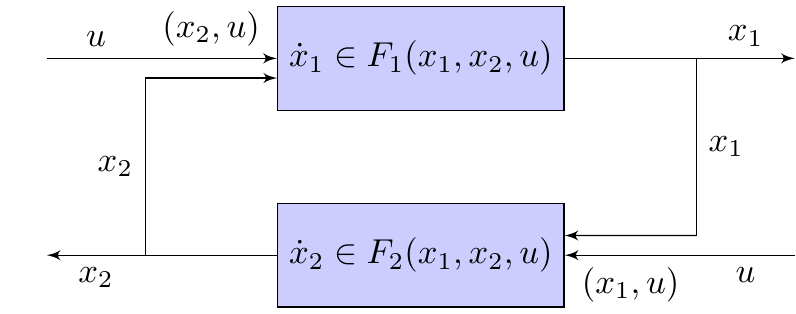}
\caption{The interconnected system in \eqref{eq:CompDiffInc}.}
\label{fig:blockdiagram}
\end{center}
\end{figure}
Since we want to combine the functions $V_1$ and $V_2$ to obtain a nonpathological ISS function $W:\R^n\to \R$ for the interconnected system~\eqref{eq:CompDiffInc}, we need the following results from non-smooth analysis.
\begin{fact}\cite[Theorem 2.6.6]{clarke3}\label{fact:1}
Consider a locally Lipschitz function $V:\R^k\to\R$  and $\sigma \in \cC^1(\R,\R)$, and define $U:=\sigma\circ V$. We have
\[
\partial U(x)=\sigma'(V(x))\partial V(x), \;\;\;\forall\,x\in \R^k,
\]
where $\sigma'(s)$ denotes the derivative of $\sigma$ at $s\in \R$.\hfill $\triangle$
\end{fact}
\begin{fact}\cite[Proposition 2.3.12]{clarke3}\label{fact:2}
Given two locally Lipschitz functions $V_1:\R^k\to \R$ and $V_2:\R^k\to \R$ consider the function $V(x):=\max\{V_1(x),V_2(x)\}$. Given any $z\in \R^k$ such that $V_1(z)=V_2(z)$, it holds that
\[
\partial V(z)\subseteq \co\{\partial V_1(z), \partial V_2(z)\}.\tag*{$\triangle$}
\]
\end{fact}
Moreover, we need this well-known comparison result, for the proof, see \cite[Theorem 3.1]{JiaMarWan96}.
\begin{fact}\label{fact:3}
Given $\chi_1,\chi_2\in \cK$ such that 
$
\chi_1\circ\chi_2(r)<r$, $\forall r>0$,
there exists a continuously differentiable $\sigma\in \cK_\infty$ with 
 $\sigma'(s)>0$ for all $s\in [0,\infty)$, such that
\begin{equation}\label{eq:SigmaFunctionBetween}
\chi_2(r)<\sigma(r),\;\;\text{and}\;\chi_1(\sigma(r))<r,\;\;\forall \;r>0.
\end{equation}
\hfill $\triangle$
\end{fact}
The geometrical intuition behind~\eqref{eq:SigmaFunctionBetween} is that the graph of the function $\sigma$ lies between the graphs of $\chi_2$ and $\chi_1^{-1}$, see for example Fig.1 in~\cite{JiaMarWan96}.
Finally, the following lemma will be used in the proof.
\begin{lemma}\label{lemma:technicalGradient}
Suppose $V_1:\Ri\to \R$ and $V_2:\Rii\to \R$ are two nonpathological functions satisfying Assumption~\ref{ass:InterLyapFuncSub}. Consider $\sigma \in \cC^1(\R_+,\R_+)$ such that $\sigma'(s)>0$ for all $s>0$ and the composite function $U_1:=\sigma \circ V_1$. Let  $W(x_1,x_2):=\max\{U_1(x_1),V_2(x_2)\}$, and consider a point $z=(z_1,z_2)\neq(0,0)$, $z_1\in \Ri$, $z_2\in \Rii$ such that $U_1(z_1)=V_2(z_2)$. It holds that 
\begin{equation}\label{eq:genGradMaxFun}
\partial W(z_1,z_2)=\co\left \{\partial \widehat U_1(z_1), \partial \widehat V_2(z_2)\right\}
\end{equation}
where 
\[
\begin{aligned}
\partial \widehat U_1(z_1)&:=\begin{pmatrix} \partial U_1(z_1)\\0\end{pmatrix}=\begin{pmatrix}  \sigma'(V_1(z_1))\partial V_1(z_1)\\0  \end{pmatrix}\;\;\text{and }\\ \partial \widehat V_2(z_2)&:=\begin{pmatrix}  0\\\partial V_2(z_2) \end{pmatrix}.
\end{aligned}
\]
\end{lemma}
\begin{proof}
Consider a point $z=(z_1,z_2)\neq(0,0)$ such that $U_1(z_1)=V_2(z_2)$, the inclusion $\partial W(z)\subseteq\co \{\partial \widehat U_1(z_1), \partial \widehat V_2(z_2)\}$ is obtained by Fact \ref{fact:2}.

 For the converse inclusion, due to convexity of $\partial W(z_1,z_2)$, it suffices to show that $\partial \widehat U_1(z_1)\subseteq \partial W(z_1,z_2)$ and $\partial \widehat V_2(z_2)\subseteq \partial W(z_1,z_2)$. We only prove the first inclusion, as the other one can be proved with a similar reasoning. To prove $\partial \widehat U_1(z_1)\subseteq \partial W(z_1,z_2)$, we note that, recalling the definition of Clarke generalized gradient  \eqref{eq:limClark} and by convexity of $\partial W(z_1,z_2)$, it suffices to show that, for each sequence $x^k_1\in \Ri$ where $U_1$ is differentiable, with $x_1^k\to z_1$ and with $v_1:=\lim_{k\to \infty}\nabla U_1(x^k_1)$, we have $\hat v:=(v_1,0)^\top\in\partial W(z_1,z_2)$. 
From \emph{1b)} and \emph{2b)} of Assumption~\ref{ass:InterLyapFuncSub}, the function $V_2$ has no local minima other than $0$ because $V_2$ is a Lyapunov function for the unperturbed system $\dot x_2\in F_2(0,x_2,0)$. Thus, considering any point $x_2\neq 0$, $V_2$ is decreasing along the solutions starting at $(0,x_2)$ with zero input. By local existence of solutions from any initial point, we have that $x_2$ cannot be a local minimum of $V_2$. Thus $z_2$ is not a local minima for $V_2$, and we can consider a sequence $x^\ell_2\to z_2$ such that $V_2(z_2)>V_2(x^\ell_2)$, for all $\ell \in \N$. Now, by continuity of $U_1$ and $V_2$, for each $\ell\in \N$, there exists  $k_\ell\in \N$ such that
\begin{equation}\label{eq:proofLemmalarger}
U_1(x^{k_\ell}_1)>V_2(x^\ell_2).
\end{equation}
Consider the sequence $\overline{x}_\ell:=\begin{pmatrix}x^{k_\ell}_1\\x^\ell_2   \end{pmatrix}\in \R^n$. We have $\overline{x}_\ell\to z=(z_1,z_2)$, and from equation~\eqref{eq:proofLemmalarger}
\[
W(\overline{x}_\ell)=\max\{U_1(x^{k_\ell}_1), V_2(x^\ell_2)\}=U_1(x^{k_\ell}_1),\;\;\forall\,\ell\in \N.
\]
Thus, $W$ is differentiable at all $\overline{x}_\ell\in \R^n$ and 
\[
\lim_{\ell\to \infty} \nabla W(\overline{x}_\ell)=\begin{pmatrix}\lim_{\ell\to \infty}\nabla U_1(x_1^{k_\ell})\\0\end{pmatrix}=\begin{pmatrix}v_1\\0\end{pmatrix}=\hat v.
\]
By definition of $\hat v$ and the generalized gradient, it follows that  $\hat v\in \partial W(z_1,z_2)$ and hence $\partial \widehat U_1(z_1)\subseteq \partial W(z_1,z_2)$.
Similarly, one can prove that $\partial \widehat V_2(z_2)\subseteq \partial W(z_1,z_2)$, and thus the equality \eqref{eq:genGradMaxFun} holds.
\end{proof}
We have now all the necessary tools to present a small gain theorem involving  nonpathological ISS functions, adapting the idea firstly proposed in~\cite{JiaMarWan96}.

\begin{thm}[Generalized Small Gain Theorem]\label{thm:SmallGainThm}
Consider the nonpathological functions $V_1,V_2$ satisfying Assumption~\ref{ass:InterLyapFuncSub} and suppose that
\begin{equation}\label{eq:SamllGainEquation}
\chi_1\circ\chi_2(r)<r,\;\;\forall r>0.
\end{equation}
Considering a function $\sigma \in \cK_\infty\cap \cC^1(\R_+,\R_+)$ satisfying property~\eqref{eq:SigmaFunctionBetween} in Fact~\ref{fact:3}, define $W:\R^{n}\to \R$ as 
\begin{equation}\label{eq:MaxLyapFunc}
W(x_1,x_2):=\max\{\sigma(V_1(x_1)), V_2(x_2)\}.
\end{equation}
Then $W$ is a nonpathological ISS function and thus system~\eqref{eq:CompDiffInc} is ISS w.r.t. $u$.
\end{thm}
\begin{proof}
We want to show that $W:\R^n \to \R$ satisfies all the conditions of Theorem~\ref{teo:mainres}. To this end, it is enough to show that
\begin{enumerate}
	\item[A)] $\sigma\circ V_1:\Ri\to\R$ is nonpathological, and $W:\R^n\to\R$ is nonpathological.
	\item[B)]There exist $\rho\in \cPD$ and $\gamma\in \cK$ such that
	\begin{equation}\label{eq:SmallGainEqToProve}
W(x)> \gamma(|u|)\;\;\Rightarrow\;\;\max  \dot{\overline{W}}_{F}(x,u)\leq -\rho(|x|).
	\end{equation}
\end{enumerate}

\emph{Proof of A):} We recall that $V_1:\Ri\to\R$ is nonpathological and $\sigma \in \cC^1(\R,\R)$ and $\sigma'(r)>0$ for all $r>0$. Defining $U_1:=\sigma \circ V_1$, by Fact~\ref{fact:1}  we have $\partial U_1(x)=\sigma'(V_1(x))\partial V_1(x)$ for all $x\in \Ri$. Moreover for any absolutely continuous function $\varphi:[0,T)\to \Ri$, by Definition~\ref{def:nonpat} we have that $\partial V_1(\varphi(t))$ is  a subset of an affine subspace orthogonal to $\dot \varphi(t)$, for almost every $t\in [0,T)$, and the same holds for $\partial U_1(\varphi(t))=\sigma'(V_1(\varphi(t))\partial V_1(\varphi(t))$. Thus $U_1:\Ri\to \R$ is nonpathological.
The non-pathology of $W:\R^n\to \R$ follows from the fact that pointwise maximum of nonpathological functions is nonpathological, as stated in Lemma~\ref{lemma: PropNonPatFunc}.\\
\emph{Proof of B):} We proceed  by considering three cases. Let us define the sets
\[
\begin{aligned}
\cO_1:=&\{(x_1,x_2)\in \R^n\;\vert\; V_2(x_2)<\sigma(V_1(x_1))\},\\
\cO_2:=&\{(x_1,x_2)\in \R^n\;\vert\; V_2(x_2)>\sigma(V_1(x_1))\},\\
\Gamma:=&\{(x_1,x_2)\in \R^n\;\vert\; V_2(x_2)=\sigma(V_1(x_1))\}.
\end{aligned}
\]
For $z=(z_1,z_2)\in \cO_1$, by continuity there exists a neighborhood $\cU$ of $z$ where $W(x)=\sigma(V_1(x_1))$, for all $x=(x_1,x_2)\in \cU$. By Fact \ref{fact:1}, we have that $\partial W(z)=\sigma'(V(z_1))\partial V_1(z_1)\times \{0\}$.
Thus $f=(f_1,f_2)\in F(z,u)$ is such that $\inp{p}{f}=a, \, \forall p \in \partial W(z)$ if and only if $f_1$ satisfies $\inp{p_1}{f_1}=a$, $\forall p_1\in \sigma'(V(z_1))\partial V_1(z_1)$. From~\eqref{eq:PPFLieDerivative}, we thus get
\begin{equation}\label{eq:WdotF_1}
\dot{\overline{W}}_{F}(z,u)=\sigma'(V(z_1))\dot {\overline{V}}_{1,F_1}(z_1,z_2,u).
\end{equation}%
Recalling that $z\in \cO_1$ and equation~\eqref{eq:SigmaFunctionBetween}, we have $\chi_1(V_2(z_2))<\chi_1(\sigma(V_1(z_1)))<V_1(z_1)$. Thus, by condition \emph{2a)} of Assumption~\ref{ass:InterLyapFuncSub}, we have from~\eqref{eq:WdotF_1} that
\begin{equation}\label{eq:IssCond1}
W(z)> \widehat{\gamma}_1(|u|) \Rightarrow \max\dot{\overline{W}}_{F}(z,u)\leq -\widehat{\rho}_1(W(z)),\;\forall z\in \cO_1,
\end{equation}
where $\widehat{\rho}_1(s):=\sigma'(\sigma^{-1}(s))\,\rho_1(\sigma^{-1}(s))$ is a positive definite function and $\widehat{\gamma}_1(s):=\sigma(\gamma_1(s))$ is class $\cK$.\\\smallskip
For $z=(z_1,z_2)\in \cO_2$, following the same reasoning (but without the complications introduced by $\sigma$), one has that
\begin{equation}\label{eq:IssCond2}
W(z)> \gamma_2(|u|) \Rightarrow\max\dot{\overline{W}}_{F}(z,u)\leq -\rho_2(W(z)),\;\forall z\in \cO_2.
\end{equation}
Before addressing $z=(z_1,z_2)\in \Gamma$, using an idea proposed in~\cite{kamalapurkar17}, we introduce the following notation motivated by definition~\eqref{eq:LieWithInput}: Given $F:\R^n\times \R^m \rightrightarrows \R^n$ and a locally Lipschitz function $V:\R^n\to \R$ we define
\[
F^V(z,u):=\left\{f\in F(z,u)\vert \exists a \in \R\,:\inp{v}{f}=a,\,\forall v\in \partial V(z)\right\}.
\]
By Definition~\ref{def:geneder}, it is clear that 
\begin{equation}\label{eq:LieNewProofofSmallGain}
\dot{\overline{V}}_{F}(z,u)=\left\{\inp{v}{f}\;\vert\; v\in \partial V(z),f\in F^V(z,u)\right\}.
\end{equation}
We continue by using the following set inclusion whose proof is postponed a few lines to avoid breaking the flow of the exposition:
\begin{equation}\label{eq:ClaimInclusion}
F^W(z,u)\subseteq F_1^{V_1}(z_1,z_2,u)\times F_2^{V_2}(z_1,z_2,u).
\end{equation}

Consider $z=(z_1,z_2)^\top\in \Gamma$ and take any $w\in \partial W(z)$, by Lemma \ref{lemma:technicalGradient}, there exist $v_1\in \partial V_1(z_1)$, $v_2\in \partial V_2(z_2)$ and $\lambda\in [0,1]$ such that
\[
w=\begin{pmatrix}
\lambda \sigma'(V_1(z_1))v_1\\(1-\lambda)v_2
\end{pmatrix}.
\] 
Consider $f =\begin{pmatrix} f_1\\ f_2\end{pmatrix}\in F^{W}(z,u)$, so that, from~\eqref{eq:ClaimInclusion}, $f_1\in F_1^{V_1}(z_1,z_2,u)$ and $f_2\in  F_2^{V_2}(z_1,z_2,u)$. Using~\eqref{eq:LieNewProofofSmallGain}, we may proceed as in~\eqref{eq:IssCond1} and~\eqref{eq:IssCond2} and use continuity of $W$ 
to get, for all $z\in \Gamma=\bd(\cO_1)\cap\bd(\cO_2)$
\begin{equation}\label{eq:NewLuca}
\begin{aligned}
W(z)> \widehat{\gamma}_1(|u|) &\Rightarrow\hskip-0.15cm \max_{\substack{f_1\in F_1^{V_1}(z,u)\\v_1\in\partial V_1(z_1)}}\hskip-0.15cm\sigma'(V_1(z_1))\inp{v_1}{f_1}\leq -\widehat{\rho}_1(W(z))\\
W(z)> \gamma_2(|u|) &\Rightarrow \max_{\substack{f_2\in F_2^{V_2}(z,u)\\v_2\in\partial V_2(z_2)}}\inp{v_2}{f_2}\leq -\rho_2(W(z)).
\end{aligned}
\end{equation}
Using~\eqref{eq:NewLuca} we finally get that $W(z)> \max\{\widehat{\gamma}_1(|u|), \gamma_2(|u|)\}$ implies
\[
\begin{aligned}
\inp{w}{f}&=\lambda\sigma'(V_1(z_1))\inp{v_1}{f_1}+(1-\lambda)\inp{v_2}{f_2}\\
&\leq-\lambda \widehat{\rho}_1(W(z))-(1-\lambda)\rho_2(W(z))\\
&\leq-\min\{\widehat{\rho}_1(W(z)), \rho_2(W(z))\}.
\end{aligned}
\]
Thus, letting $\gamma(s):=\max\{\widehat{\gamma}_1(s), \gamma_2(s)\}$ and $\rho(s):=\min\{\widehat{\rho}_1(s),\rho_2(s)\}$, we have
\begin{equation}\label{eq:ISScond3}
W(z)> \gamma(|u|)\;\Rightarrow \max\dot{\overline{W}}_{F}(z,u)\leq- {\rho}(W(z)),\;\forall z\in \Gamma.
\end{equation}
Collecting \eqref{eq:IssCond1}, \eqref{eq:IssCond2} and \eqref{eq:ISScond3} we can conclude \eqref{eq:SmallGainEqToProve}, and prove item B).

We complete the proof by proving~\eqref{eq:ClaimInclusion}. To this, take any $f\in F^W(z,u)$. By definition of $F$ in~\eqref{eq:CompDiffInc}, we have that $f=\begin{pmatrix} f_1 \\ f_2\end{pmatrix}$ for some $f_1\in F_1(z_1,z_2,u)$ and $f_2\in F_2(z_1,z_2,u)$. By Lemma \ref{lemma:technicalGradient} and Fact \ref{fact:1}, for any $v_1\in \partial V_1(z_1)$, the vector $w=\begin{pmatrix} \sigma'(V_1(z_1))v_1\\0 \end{pmatrix}\in \partial W(z)$ and thus 
\[
\inp{w}{f}=\sigma'(V_1(z_1))\inp{v_1}{f_1}.
\]
By varying $v_1$ in $\partial V_1(z_1)$ and recalling that $f\in F^W(z,u)$ (and thus $\inp{w}{f}$ is constant for all $w\in \partial W(z)$), we obtain that $f_1\in F_1^{V_1}(z_1,z_2,u)$, that is
$
\inp{v_1}{f_1}\;\;\text{constant w.r.t. } v_1\in \partial V_1(z_1).
$ 
The same reasoning applies to $f_2$, considering a vector $w=\begin{pmatrix}0\\v_2\end{pmatrix}\in \partial W(z)$, with $v_2\in\partial V_2(z_2)$, concluding the proof of the claim.
\end{proof}
\begin{oss}
The idea of analyzing the derivative of the composite function $W$ in the three sets $\cO_1,\cO_2,\Gamma$, appeared firstly in \cite{JiaMarWan96}, and is the common idea of many results on small-gain theorems for interconnected systems, see for example \cite{LibNes14} or \cite{Ito19}. The analysis in $\cO_1$ and $\cO_2$ was straightforward, but because of non-differentiability of $V_1$ and $V_2$, the analysis in the set $\Gamma$ is different from \cite{JiaMarWan96}. In particular  the additional tools from nonsmooth analysis have been used to study the Lie-derivative of $W$ along $F$ on the set $\Gamma$. \hfill $\triangle$

\end{oss}
\subsection{Cascade System}
We now apply Theorem \ref{teo:mainres} to \emph{cascade interconnections} as in Figure~\ref{fig:CascadeDiagram}.
More precisely, given two maps $F_1: \R^{n_1}\times \R^{n_2}\times \R^m\rightrightarrows \R^{n_1}$, and $F_2: \R^{n_2}\times \R^{m}\rightrightarrows \R^{n_2}$  we consider the \emph{cascade} system defined as follow: 
\begin{subequations}\label{eq:cascadesys}
\begin{align}
 \dot x_1 &\in F_1(x_1,x_2,u),\label{eq:cascade1} \\
\dot x_2 &\in F_2(x_2,\, u).\label{eq:cascade2}
\end{align}
\end{subequations}
Defining again $n:=n_1+n_2$ we will write $ F: \R^n\times\R^{m}\rightrightarrows \R^n$ defined by 
\[
F(x_1,x_2,u):=\begin{pmatrix}
F_1(x_1,x_2,u)\\
F_2(x_2,\,u)
\end{pmatrix}.
\]
\begin{figure}
\begin{center}
\includegraphics[scale=1]{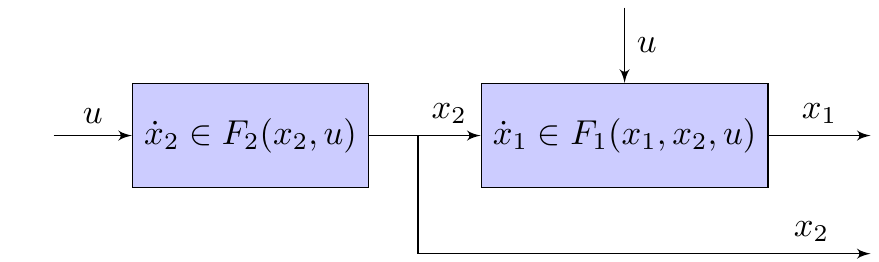}
\caption{The cascade system in \eqref{eq:cascadesys}.}
\label{fig:CascadeDiagram}
\end{center}
\end{figure}

The cascade system \eqref{eq:cascadesys} can be seen as a system of the form \eqref{eq:CompDiffInc}  where  $F_2$ does not depend on $x_1$, see also Figure~\ref{fig:CascadeDiagram}. Therefore Theorem \ref{thm:SmallGainThm} can be applied with $\chi_1\equiv 0$ and condition~\eqref{eq:SamllGainEquation} holds for any $\chi_2\in \cK$. On the other hand, the cascade structure allows us to construct a different ISS-Lyapunov function, based on two non-smooth ISS-Lyapunov functions associated with each subsystem. The function that we construct is in the so-called \emph{sum-separable} form, that has some clear advantages with respect to the \emph{max-separable} form~\eqref{eq:MaxLyapFunc} in Theorem~\ref{thm:SmallGainThm}, see~\cite{ItoJin13} and references therein for a thorough discussion. In particular, the sum-separable architecture preserves regularity, and in our setting also leads to a more direct proof of ISS of the cascade interconnection.

 In the following we adapt, in the framework of differential inclusions and nonpathological functions,  the proof technique proposed firstly in~\cite{SonTeel95}. More specifically, we assume that both subsystems admit an ISS Lyapunov function, using the formulation~\eqref{eq:cond2eqival} in Remark~\ref{rmk:equiv}. Similar constructions can be found in~\cite{TanTeel15} and~\cite{tanwani2018}.
\begin{assumption}\label{ass:casc}
The following conditions hold for system~\eqref{eq:cascadesys}:
\begin{enumerate}[leftmargin=*,label=\textbf{(A.\arabic*)}]
\item \label{Item:CascadeA.1} \emph{ISS of \eqref{eq:cascade2}}. There exist a nonpathological function $V_2: \R^{n_2} \to \R$ and $\underline \alpha_2$ $\overline \alpha_2$, $\rho_2\in\mathcal{K}_\infty$ and $\gamma_2\in \cK$ satisfying
\begin{align*}
\underline \alpha_2(|x_2|)&\leq V_2(x_2) \leq \overline \alpha_2(|x_2|), \hskip0.4cm \forall x_2 \in \R^{n_2}, \\
\max \dot{\overline{V}}_{2,F_2}(x_2,u) &\leq -\rho_2(V_2(x_2))+\gamma_2(|u|), \;
\end{align*}
for all $x_2 \in \R^{n_2}$ and for all $u\in \R^{m}$.
\item \label{Item:cascadeA.2} \emph{ISS of \eqref{eq:cascade1}.} There exist a nonpathological function $V_1: \R^{n_1} \to \R$ and $\underline \alpha_1$ $\overline\alpha_1$, $\rho_1$, $\gamma_1 \in \cK_\infty$  satisfying
\begin{equation*}
\begin{aligned}
&\underline\alpha_1(|x_1|)\leq V_1(x_1) \leq \overline{\alpha}_1(|x_1|), \hskip0.4cm \forall x_1 \in \R^{n_1}, \\
&\begin{aligned}\max \dot{\overline{V}}_{1,F_1}(x_1,x_2,u) \leq  &-\rho_1(V_1(x_1))\\&+\gamma_1(V_2(x_2))+\gamma_2(|u|),\end{aligned}
\end{aligned}
\end{equation*}
for all $x_1 \in \R^{n_1}$, $x_2\in \R^{n_2}$ and all $u\in \R^m$.
\item \label{Item:CascadeA.3} Defining $\overline{\nu}(s):=\gamma_1(s)/\rho_2(s)$, there exists a scalar $M>0$ such that 
\[
\lim_{s\to 0^+}\overline{\nu}(s)=\lim_{s\to0^+}\frac{\gamma_1(s)}{\rho_2(s)} \leq M.
\]
In other words, $\gamma_1(s) \in O(\rho_2(s))$ as $s \to 0^+$.\hfill $\triangle$
\end{enumerate}
\end{assumption}

\begin{oss}[Tightness of Assumption \ref{ass:casc}] 
Condition~\ref{Item:CascadeA.3}, which is used in the construction of $W$ in the proof of Proposition \ref{prop:cascprop}, is not restrictive: if it does not hold it is possible to modify the function $V_1$ in such a way that it holds, following the same idea proposed in \cite{SonTeel95}. Due to this fact, Proposition~\ref{prop:cascprop} establishes that when system~\eqref{eq:InterconnectedDiffInc} is in the cascade-form presented in equation~\eqref{eq:cascadesys}, it suffices to have ISS-Lyapunov functions (satisfying the Lie-derivative conditions presented in \ref{Item:CascadeA.1} and \ref{Item:cascadeA.2} ) for each subsystem, to conclude ISS of the interconnected system. In this context, the \emph{small gain condition} required in the general construction of Theorem~\ref{thm:SmallGainThm} is somehow trivially satisfied.\hfill $\triangle$
\end{oss}
Using Assumption~\ref{ass:casc}, we can construct a nonpathological Lyapunov function for the cascade system \eqref{eq:cascadesys}, by adapting a Lyapunov design developed in~\cite{SonTeel95} and~\cite{tanwani2018}.
\begin{prop}\label{prop:cascprop}
Consider the cascade system \eqref{eq:cascadesys}, and suppose that Assumption~\ref{ass:casc} holds. There exists a continuous and nondecreasing function  $\nu:\R_+\to\R_+$ satisfying $\nu(s)\geq 4 \overline{\nu}(s)$, for all $s \in \R_+$. Moreover, the function
\begin{equation}
 W(x_1,x_2):=\int_{0}^{V_2(x_2)} \nu(s)ds +V_1(x)
\end{equation}
is a nonpathological ISS functions for system~\eqref{eq:cascadesys}; that is there exist $\underline{\alpha}, \overline{\alpha}\in \mathcal{K}_\infty$ such that
\begin{equation}\label{eq:posdefcasc}
\underline{\alpha}(|(x_1,x_2)|)\leq  W(x_1,x_2)\leq \overline{\alpha}(|(x_1,x_2)|)
\end{equation}
for all $(x_1,x_2)\in \R^{n_1}\times \R^{n_2}$, and there exist $\rho\in\mathcal{K}_\infty$ and $\gamma\in \cK$ such that
\begin{equation}\label{eq:casclyap}
\max \dot{\overline{W}}_{ F}(x_1,x_2,u)\leq - \rho(W(x_1,x_2))+ \gamma(|u|),
\end{equation}
for all $(x_1,x_2)\in \R^{n_1}\times \R^{n_2}$ and for all $u \in \R^m$. 
\end{prop}
\begin{proof}
The existence of function $\nu:\R_+\to \R_+$ under~\ref{Item:CascadeA.3} of Assumption~\ref{ass:casc} is established in~\cite[Lemmas 1 and 2]{SonTeel95}. Introduce the function $\ell :\R_+ \to \R_+$ defined by
\begin{equation*}
\ell(s):=\int_0^{s} \nu(r)dr,\;\;\;\forall\;s\in \R_+.
\end{equation*}
Since $\nu(s)\geq 4\overline \nu(s)>0$, $\forall s>0$, then $\ell$ is a class $\mathcal{K}_\infty$ function. Moreover $\ell\in \cC^1(\R_+,\R_+)$. We can thus rewrite
\[
 W(x_1,x_2)=\ell \circ V_2(x_2)+ V_1(x_1).
\]
Non-pathology of $W$ follows from Proposition~\ref{lemma:propertiesofppf} and Fact~\ref{fact:1} since $\ell$ is $\cC^1$ by construction.
Moreover, the functions $\underline{\alpha}$ and $\overline{\alpha}$ of equation \eqref{eq:posdefcasc} are easily constructed as $\underline{\alpha}(s):=\int_0^{\underline \alpha_2(s)}\nu(r)dr+\underline \alpha_1(s)$ and $\overline{\alpha}(s):=\int_0^{\overline \alpha_2(s)}\nu(r)dr+\overline \alpha_1(s)$.\\
Let us now define $U_2:=\ell \circ V_2$; noting that $\ell'(s)=\nu(s)$ and using Fact~\ref{fact:1}, we have  that
$
\dot{\overline{U}}_{2,F_2}(x_2,u)=\nu(V_2(x_2)) \dot{\overline{V}}_{2,F_2}(x_2,u).
$
Recalling \ref{Item:CascadeA.1}, we can write
\begin{equation}\label{eq:CascadePrelimEq1}
\begin{aligned}
\max \dot{\overline{U}}_{2,F_2}(x_2,u)\leq&-\nu(V_2(x_2))\rho_2(V_2(x_2))\\&+\nu(V_2(x_2))\gamma_2(|u|), \;\;
\end{aligned}
\end{equation}
$\forall\,x_2 \in \R^{n_2}, \;\forall\,u \in \R^m$. Defining $\theta(s):=\rho_2^{-1}(2 \gamma_2(s))$, we prove the following inequality
\begin{equation}\label{eq:cascineq2}
\begin{aligned}
\max \dot{\overline{U}}_{2,F_2}(x_2,u)\leq &-\frac{1}{2} \nu(V_2(x_2))\rho_2(V_2(x_2))\\&+\nu(\theta(|u|))\gamma_2(|u|).
\end{aligned}
\end{equation}
Indeed, by~\eqref{eq:CascadePrelimEq1}, if $\gamma_2(|u|) \leq \frac{1}{2}\rho_2(V_2(x_2))$, \eqref{eq:cascineq2} trivially holds. Otherwise, we see that
\[ 
\gamma_2(|u|) \geq \frac{1}{2}\rho_2(V_2(x_2)) \Leftrightarrow V_2(x_2) \leq \rho_2^{-1}(2 \gamma_2(|u|))=\theta(|u|),
\]
and by the nondecreasing property of $\nu$, inequality \eqref{eq:cascineq2} holds.
Before proceeding to proving~\eqref{eq:casclyap} we observe the following equality 
\begin{equation}\label{eq:charater}
\dot{\overline{W}}_{F}(x_1,x_2,u)=\dot{\overline{U}}_{2,F_2}(x_2,u)+\dot{\overline{V}}_{1,F_1}(x_1,x_2).
\end{equation}
To show~\eqref{eq:charater}, we recall that any locally Lipschitz function $G: \R^{n_1}\times \R^{n_2} \to \R$ defined by $G(x_1,x_2)=G_1(x_1)+G_2(x_2)$, satisfies
\begin{equation}
\partial G(x_1,x_2)= \left \{ \begin{pmatrix}
v_1\\
v_2
\end{pmatrix}
  \vert\, v_1 \in \partial G_1(x_1), v_2 \in \partial G_2(x_2) \right \}
\end{equation}
and thus, using definition~\eqref{eq:LieWithInput}, we obtain~\eqref{eq:charater}.
From~\ref{Item:cascadeA.2},~\eqref{eq:cascineq2} and~\eqref{eq:charater}, we have
\begin{equation*}
\begin{aligned}
\hskip-0.05cm\max \dot{\overline{W}}_{ F}(x_1,x_2,u)&\leq -\rho_1(V_1(x_1))+\gamma_1(V_2(x_2))+\gamma_2(|u|)\\&\hskip-0.1cm-\frac{1}{2} \nu(V_2(x_2))\rho_2(V_2(x_2))\hskip-0.1cm+\hskip-0.1cm\nu(\theta(|u|))\gamma_2(|u|).
\end{aligned}
\end{equation*}
From the assumption $\nu(s) \geq 4\frac{\gamma_1(s)}{\rho_2(s)}$ for all $s\in \R_+$,  and following \cite[Lemma 10]{Kellett2014}, we finally  have
\[
\begin{aligned}
\max \dot{\overline{W}}_{F}(x_1,x_2)&\leq -\gamma_1(V_2(x_2))-\rho_1(V_1(x_1))\\&\;\;\;\;+\nu(\theta(|u|))\gamma_2(|u|)+\gamma_2(|u|)\\
&\leq- \rho(W(x_1,x_2))+\gamma(|u|),
\end{aligned}
\]
where we have defined
\[
\begin{aligned}
\gamma(s)&:=\left(\nu(\theta(s))+1\right)\gamma_2(s),\\
\rho(s)&:=\min \left\{\rho_1(\frac{1}{2}s), \gamma_1(\frac{1}{2}\ell^{-1}(s))\right \}.
\end{aligned}
\]
\end{proof}

\section{Feedback Stabilization of a 2-mode State-dependent Switched System}\label{sec:appStab}
We now use the tools developed in the previous section to solve an output feedback stabilization problem for a class of switched systems with two modes. In particular, we consider the state dependent switched system defined as
\begin{equation}\label{eq:two-modeSystem}
\cS:\begin{cases}
\dot x=\begin{cases}
f_1(x)+g(x)u\;\;\text{if}\,x\in X_1:=\{q(x)\geq 0\},\\
f_2(x)+g(x)u\;\;\text{if}\,x\in X_2:=\{q(x)\leq0\},\\
\end{cases}\\
y=h(x),
\end{cases}
\end{equation}
where $x\in \R^n$ and $u\in \R^m$. The basic assumptions we impose on the system~\eqref{eq:two-modeSystem} are the following:

\begin{assumption}\label{ass:mainassumpCONTr}
The data in~\eqref{eq:two-modeSystem} is such that
\begin{enumerate}[leftmargin=*,label=\alph*)]
\item $f_1,f_2\in \cC^1(\R^n,\R^n)$;
\item $g\in \cC^1(\R^n, \R^{n\times m})$;
\item $h\in \cC^1(\R^n,\R^p)$;
\item $q\in \cC^1(\R^n,\R)$ and $\cX=\{X_i,\R^n\}_{i\in \{1,2\}}$ form a proper partition of $\R^n$ (recall Definition~\ref{def:Ncov});\label{item:conditionsPartition}
\item $q(0)\geq0\;\Rightarrow f_1(0)=0$ and $q(0)\leq0\;\Rightarrow f_2(0)=0$;
\item There exists $\beta_g\in \cC(\R,\R_+)$, $\beta_g(s)\geq 0$ for all $s\geq 0$, such that $\|g(x)\|\leq \beta_g(|x|)$, for all $x\in \R^n$,
\item There exists $\beta_f \in\cK_\infty$ such that $|f_1(x)-f_2(x)|\leq \beta_f(|x|)$, for all $x\in \R^n$.\hfill $\triangle$
\end{enumerate}
\end{assumption}
\begin{myexample}[Regular Values and Partitions]
Consider $q\in \cC^1(\R^n,\R)$ such that $0$ is a regular value of $q$, i.e. $\nabla q(x)\neq 0$ for all $x$ satisfying $q(x)=0$; then condition~\ref{item:conditionsPartition} of Assumption~\ref{ass:mainassumpCONTr} is satisfied. Indeed, by the Implicit Function Theorem, $X_0:=\{x\in \R^n\;\vert\;q(x)=0\}$ is a $(n-1)$-dimensional $\cC^1$ manifold, and hence $X_0$ has Lebesgue measure $0$. Let us prove that $\{X_i,\R^n\}_{i\in \{1,2\}}$ is a proper partition, with $X_1$ and $X_2$ defined as in~\eqref{eq:two-modeSystem}. First of all,  $X_1\cup X_2=\R^n$ and $\bd(X_i)\subseteq X_0$, for any $i\in \{1,2\}$: in fact, if $q(x)>0$ (resp. $q(x)<0$), by continuity of $q$ it holds that $x\in \inn(X_1)$ (resp. $x\in\inn(X_2)$). It remains to prove that $\overline{\inn(X_i)}=X_i$, for any $i\in \cI$. Consider w.l.o.g. $i=1$; the inclusion $\overline{\inn(X_i)}\subseteq X_i$ is trivial. Let us consider $x\in X_1$, i.e. $q(x)\geq 0$. If $q(x)>0$, then $x\in \inn(X_1)\subset \overline{\inn(X_1)}$. If $q(x)=0$, by assumption $\nabla q(x)\neq 0$, therefore $x$ is neither a maximum nor a minimum. Thus there exists a sequence $x_k\to x$ such that $q(x_k)>0$, $\forall \,k\in \N$, and hence, $x\in \overline{\inn(X_1)}$. \hfill $\triangle$
\end{myexample}
\begin{myexample}[Switched Linear Case]\label{rmk:Linear}
As a simple paradigm, one can think of a state-dependent switched linear system, such as
\[
f_i(x)=A_ix,\;\;\;g(x)\equiv B,\;\;\;h(x)=Cx,
\]
where $A_i\in \R^{n\times n}$ for $i\in \{1,2\}$, $B\in \R^{n\times m}$ and $C\in \R^{p\times n}$.
Regarding the function $q\in \cC^1(\R^n,\R)$, the simplest non-trivial cases are the halfspace partitions or the symmetric conic partitions, described respectively by the functions
\[
q_v(x):=\inp{v}{x}\;\;\;\text{or}\;\;\;q_Q(x):=x^\top Q x,
\]
for some $v\in \R^n$, or $Q\in \Sym(\R^{n\times n}):=\{S\in\R^{n\times n}\;\vert\;S^\top=S\}$, $Q$ is neither negative, nor positive semi-definite. These cases satisfy Assumption~\ref{ass:mainassumpCONTr}, by selecting
\[
\beta_g(s):=\|B\|,\;\;\text{and } \;\beta_f(s):=\|A_1-A_2\|s.\tag*{$\triangle$}
\]
\end{myexample}

Under Assumption~\ref{ass:mainassumpCONTr}, we design next an observer-based controller for system~\eqref{eq:two-modeSystem} of the form
\begin{equation}\label{eq:two-modeController}
\cC:\begin{cases}
\dot z=
\begin{cases}
f_1(z)+g(z)u+\ell_1(y-h(z))\;\;\;\text{if }z\in X_1,\\
f_2(z)+g(z)u+\ell_2(y-h(z))\;\;\;\text{if }z\in X_2,
\end{cases}\\
u=k(z),
\end{cases}
\end{equation}
where $\ell_1,\ell_2\in \cC^1(\R^p,\R^n)$, and the globally Lipschitz map $k:\R^n\to\R^m$ are  design parameters.\footnote{The globally Lipschitz assumption on $k:\R^n\to \R^m$ can be relaxed by  asking that $\exists \alpha_k\in \cK_\infty$ such that $|k(x)-k(y)|\leq \alpha_k(|x-y|)$, for all $x,y\in \R^n$.} 
The design of globally Lipschitz feedback laws is a rather common occurrence in stabilization problems for various kinds of nonlinear systems and in particular the design methods in~\cite[Chapters 13, 14]{khalil2002nonlinear} can be adapted to meet this requirement. Moreover, when restricting the attention to initial states in a compact set, it is possible to develop semiglobal results to allow for locally Lipschitz feedbacks, as explained in~\cite{TeelPra94}.
We consider the interconnected system \eqref{eq:two-modeSystem}-\eqref{eq:two-modeController}, and in particular its Filippov regularization, which can be written as follows
\begin{subequations}\label{eq:coupledRegularization}
\begin{align} 
&\dot x\in \co\left \{f_i(x)\;\vert\;i\in \cI_\cX(x)\right\}+g(x)k(z)=:\wt F_x(x,z),\label{eq:z<-xsystem}\\
&\begin{aligned}
\dot z\in &\co\left\{  f_i(z)+\ell_i(h(x)-h(z))\;\vert\;i\in \cI_\cX(z)\right\}\\&+g(z)k(z)=:\wt F_z(x,z)\label{eq:x-zSystem},
\end{aligned}
\end{align}
\end{subequations}
where the function $\cI_\cX$ is defined as in~\eqref{eq:funcI}.\\
The maps $\wt F_x,\wt F_z:\R^n\times \R^n\rightrightarrows \R^n$ satisfy the conditions of Assumption \ref{Assump:Main}: they have non-empty, compact and convex values, they are locally bounded, upper semicontinuous with respect to the states ($x$ and $z$ respectively) and continuous with respect to the inputs ($z$ and $x$ respectively). We can thus conclude local existence of solutions for the systems~\eqref{eq:z<-xsystem},~\eqref{eq:x-zSystem} using Proposition~\ref{prop:Existence}.

To design~\eqref{eq:two-modeController}, we first characterize stability of the interconnection~\eqref{eq:coupledRegularization}. To this end, we perform the change of coordinates $(x,z)\mapsto (x,e):=(x,x-z)$ and we construct the Filippov regularization of the corresponding dynamics, resulting in
\vspace{-0.1cm}
\begin{subequations}\label{eq:StateAndErros}
\begin{align} 
&\dot x\in F_x(x,e):=\wt F_x(x,x-e)\label{eq:xSystem}\\
&\dot e\in F_e(x,e),\label{eq:eSystem}
\end{align}
\end{subequations}
where the map $F_{e}:\R^n\times \R^n\rightrightarrows \R^n$ is defined as the Filippov regularization of the discontinuous map $f_{e}(x,e):=$
\small
\begin{equation}\label{eq:ErrorDiscDynamic}
\begin{aligned}
&f_1(x)-f_1(z)-\ell_1(h(x)-h(z))+\widetilde{g}(x,z)\;\text{if } q(x)\geq 0, q(z)\geq 0,\\
&f_2(x)-f_1(z)-\ell_1(h(x)-h(z))+\widetilde{g}(x,z)\;\text{if } q(x)\leq 0, q(z)\geq 0,\\
&f_1(x)-f_2(z)-\ell_2(h(x)-h(z))+\widetilde{g}(x,z)\;\text{if } q(x)\geq 0, q(z)\leq 0,\\
&f_2(x)-f_2(z)-\ell_2(h(x)-h(z))+\widetilde{g}(x,z)\;\text{if } q(x)\leq 0, q(z)\leq 0,
\end{aligned}
\end{equation}
\normalsize
with $\widetilde{g}(x,z):=(g(x)-g(z))k(z)$. 

In our construction, we first use Theorem~\ref{thm:SmallGainThm} to ensure ISS of~\eqref{eq:xSystem} based on two functions $V_1,V_2$, each of them associated to a mode.
\begin{property}\label{lemm:ISSx}
There exist $V_1,V_2\in \cC^1(\R^n,\R)$, and $\underline\psi_x,\overline\psi_x,\rho_x,\alpha_x\in\cK_\infty$, such that, for each $x\in \R^n$,  $(-1)^{i-1}q(x)> 0$ implies
\begin{subequations}
\begin{align} 
&\underline\psi_x(|x|)\leq V_i(x)\leq \overline\psi_x(|x|),\label{eq:xDynamicContinuity}\\
&|\nabla V_i(x)|\leq \rho_x(|x|),\label{eq:xDynamicBoundGradient}\\
&\inp{\nabla V_i(x)}{f_i(x)+g(x)k(x)}\leq -\alpha_x(|x|)\label{eq:xDynamicLyapunovDecrease}.
\end{align}
\end{subequations}
Moreover, there exists a function $\theta_x\in\cK_\infty$  such that 
\begin{equation}
\theta_x(s)\leq\frac{\alpha_x(s)}{\beta_g(s)\rho_x(s)}, \;\;\forall\;s\in \R_+.\label{eq:xDynamicStrangeFunct}
\end{equation}
Finally, defining
\begin{equation}
V_x(x):=V_i(x),\;\;\;\;\text{if }x\in X_i\;\;i=1,2,
\end{equation}
we suppose that $V_x$ is continuous, that is,
\begin{equation}\label{eq:continuityProperty1}
(q(x)=0)\;\Rightarrow V_1(x)=V_2(x),
\end{equation}
and there exist functions $\gamma^q_{x}\in \cK$ and $\alpha^q_{x}\in \cPD$ such that for all $x\in \R^n$ satisfying $q(x)=0$, it holds that
\begin{equation}\label{eq:xDynamicExplicitLie}
( |x|\geq \gamma^q_{x}(|e|)\,)\Rightarrow\;\max \dot{\overline{V}}_{x,F_x}(x,e)\leq - \alpha^q_x(|x|).
\end{equation}
\end{property}
Based on Property~\ref{lemm:ISSx}, we may prove the next result.
\begin{prop}\label{prop:xISSformE}
Under Property~\ref{lemm:ISSx}, there exists  $\widehat \alpha_x\in \cPD$ and $\widehat \gamma_x\in \cK$ such that
\[
(|x|\geq \widehat\gamma_x(|e|)\,)\;\Rightarrow\; \max \dot{\overline{V}}_{x,F_x}(x,e)\leq - \widehat\alpha_x(|x|)
\]
and thus system \eqref{eq:xSystem} is ISS w.r.t. $e$.
\end{prop}
\begin{proof}
First of all we rewrite system~\eqref{eq:xSystem} as 
\[
\dot x\in \co\left \{f_i(x)\;\vert\;i\in \cI_\cX(x)\right\}+g(x)k(x)-g(x)\left [k(x)-k(z)\right].
\]
Let us note that equation~\eqref{eq:xDynamicContinuity} assures continuity of the function $V_x$, and thus $V_x$ is a piecewise $\cC^1$ function with respect to $\cX$,  in the sense of Definition~\ref{defn:piecewise}.
Consider first a point $x\in \inn(X_i)$ for some $i\in \{1,2\}$, that is an $x\in \R^n$ such that $(-1)^{i-1}q(x)> 0$. We have
\[
F_x(x,e)=\left\{f_i(x)+g(x)k(x)-g(x)[k(x)-k(z)]\right\}=:\{\wt f_i\},
\]
and thus by equation~\eqref{eq:xDynamicLyapunovDecrease} it follows that, for $x\in X_i$,
\[
\begin{aligned}
\inp{\nabla V_x(x)}{\wt f_i}&=\inp{\nabla V_i(x)}{\wt f_i}\\&\leq - \alpha_x(|x|)+|\nabla V_i(x)|\|g(x)\||k(x)-k(z)|\\&\leq - \alpha_x(|x|)+\rho_x(|x|)\beta_g(|x|)L_k|x-z|,
\end{aligned}
\]
where $\beta_g$ comes from Assumption \ref{ass:mainassumpCONTr}, $L_k>0$ is the Lipschitz constant of the map $k:\R^n\to\R$ and $\rho_x$ is given by equation~\eqref{eq:xDynamicBoundGradient}.
Choosing $0<\varepsilon<1$, we have
\[
\begin{aligned}
\inp{\nabla V_x(x)}{\wt f_i}\leq &-(1-\varepsilon)\alpha_x(|x|) \\&-\varepsilon \alpha_x(|x|)+\rho_x(|x|)\beta_g(|x|)L_k|x-z|
\end{aligned}
\]
and thus, for $i\in \{1,2\}$, and each $x\in X_i$,
\[
\inp{\nabla V_x(x)}{\wt f_i}\leq -(1-\varepsilon)\alpha_x(|x|),\;\text{if}\;|e|\leq \frac{\varepsilon\alpha_x(|x|)}{L_k\rho_x(|x|)\beta_g(|x|)}.
\]
Thanks to~\eqref{eq:xDynamicStrangeFunct}, the function $\widehat\theta_x(s):=\frac{\varepsilon \theta_x(s)}{L_k}\leq\frac{\varepsilon\alpha_x(s)}{L_k\rho_x(s)\beta_g(s)}$ is of class $\cK_\infty$. Defining $\alpha^c_x:=(1-\varepsilon)\alpha_x$ and $ \gamma^c_x:=\widehat \theta_x^{-1}$, by arbitrariness of $i\in \{1,2\}$, the previous inequality implies that,  for any $x\in \inn(X_1)\cup\inn(X_2)$, 
\begin{equation}\label{eq:PropLieInterior}
(|x|\geq  \gamma^c_x(|e|)\,)\;\Rightarrow\;\max \dot{\overline{V}}_{x,F_x}(x,e)\leq - \alpha^c_x(|x|).
\end{equation}
Consider now a point $x\in \partial X_1\cup\partial X_2$. By definition of the proper partition $\{X_1,X_2\}$, we have $q(x)=0$, and thus implication~\eqref{eq:xDynamicExplicitLie} holds.
Collecting~\eqref{eq:xDynamicExplicitLie} and~\eqref{eq:PropLieInterior} we obtain that, for all $x,e\in \R^n$,
\[
(|x|\geq  \widehat\gamma_x(|e|)\,)\;\Rightarrow\;\max \dot{\overline{V}}_{x,F_x}(x,e)\leq - \widehat\alpha_x(|x|),
\]
where $\widehat\gamma_x(s):=\max\{\gamma^c(s),\gamma^q_x(s)\}$ and $\widehat\alpha_x(s):=\min\{\alpha^c_x(s),\alpha^q_x(s)\}$, concluding the proof.
\end{proof}

\begin{oss}[Clarke derivative based condition]\label{rmk:ClarkeRelaxation}
It is possible to obtain a corollary of Proposition~\ref{prop:xISSformE}, based on the Clarke derivative, as in Definition~\ref{def:geneder}. To this end, it is sufficient to replace implication~\eqref{eq:xDynamicExplicitLie} in Property~\ref{lemm:ISSx}, with the following:
\begin{enumerate}[leftmargin=*,label={(CL.\arabic*)}]
\item For all $x\in \R^n$ such that $q(x)= 0$, for each $i\in \{1,2\}$,
\[
\inp{\nabla V_i(x)}{f_{3-i}(x)+g(x)k(x)}\leq -\alpha_x(|x|),\label{item:Clarkebasedcond}
\]
where $\alpha_x\in \cK_\infty$ satisfies also~\eqref{eq:xDynamicLyapunovDecrease}. 
\end{enumerate}
The proof carries over straightforwardly, recalling the inclusion~\eqref{eq:setdiffinclusion}.\hfill$\triangle$
\end{oss}
\setcounter{myexample}{2}
\begin{myexample}[Continued]\label{rmk:Linear2}
In the switched linear case of Example~\ref{rmk:Linear}, Property~\ref{lemm:ISSx} can be guaranteed with quadratic functions $V_i(x):=x^\top P_i x$, $i\in \{1,2\}$. Indeed, since the partitions given by $q_v$ (or $q_Q$) are conic, i.e. $X_i$ is a cone for each $i\in \{1,2\}$, we can look for Lyapunov functions \emph{homogeneous of degree $2$}, see~\cite{Ros92} and the extension~\cite{TunTeel06}. From now on we focus on the case $q_Q(x)=x^\top Q x$. The half-space partition case (i.e. considering $q_v(x)=\inp{v}{x}$) can be developed analogously to~\cite{heemels2007input}.

Considering $q_Q(x)$, it suffices to find  $K\in  \R^{m\times n}$, positive definite matrices  $P_1,P_2\in \R^{n\times n}$, $\mu_{12},\mu_{21},\mu_Q\in \R$, $a_x>0$ and $\mu_1,\mu_2\geq0$ such that \small
\begin{subequations}
\begin{align}
P_1-P_2=\mu_Q Q&;\label{eq:ContinuityLinear}\\
 \mu_1Q+P_1(A_1+BK)+(A_1+BK)^\top P_1+a_xI&\prec0;\label{eq:Sprocd1}\\
-\mu_2Q+P_2(A_2+BK)+(A_2+BK)^\top P_2+a_xI&\prec0;\label{eq:Sprocd2}\\
\mu_{12} Q+P_1(A_2+BK)+(A_2+BK)^\top P_1+a_xI&\prec0;\label{eq:FInsler1}\\
\mu_{21} Q+P_2(A_1+BK)+(A_1+BK)^\top P_2+a_xI&\prec0.\label{eq:Finsler2}
\end{align}
\end{subequations}\normalsize
Then all the conditions of Property~\ref{lemm:ISSx} hold with $V_x(x):=x^\top P_ix$, if $x\in X_i$.
Indeed, first we note that \eqref{eq:ContinuityLinear} implies~\eqref{eq:continuityProperty1} of Property~\ref{lemm:ISSx}. Moreover, we can define
\[
\underline{\lambda_x}:=\min_{i\in \{1,2\}}\{\lambda_{\text{min}}(P_i)\},\;\;\;\overline{\lambda_x}:=\max_{i\in \{1,2\}}\{\lambda_{\text{max}}(P_i)\},
\]
 where $\lambda_{\max}(P),\lambda_{\min}(P)$ represent respectively the largest and the smallest eigenvalues of a positive definite matrix $P\succ0$. The bound functions in~\eqref{eq:xDynamicContinuity} and~\eqref{eq:xDynamicBoundGradient} of Property~\ref{lemm:ISSx} are thus obtained by defining
\[
\underline{\psi_x}(s):=\underline{\lambda_x}s^2,\;\;\overline{\psi_x}(s):=\overline{\lambda_x}s^2,\;\;\rho_x(s):=2\overline{\lambda_x}s.
\]
Via the S-Procedure, equation~\eqref{eq:Sprocd1} implies
\[
 x^\top (P_1(A_1+BK)+(A_1+BK)^\top P_1)x<-a_x |x|^2,\;
\]
if  $x^\top Qx\geq 0$ and equation~\eqref{eq:Sprocd2} implies
\[
\hskip-0.1cmx^\top (P_2(A_2+BK)+(A_2+BK)^\top P_2)x<-a_x |x|^2,\;{}
\]
if $x^\top Qx\leq 0$. We have thus proved~\eqref{eq:xDynamicLyapunovDecrease} of Property~\ref{lemm:ISSx} with $\alpha_x(s):=a_x s^2$. 
Similarly,  using Finsler's Lemma, equations~\eqref{eq:FInsler1} and \eqref{eq:Finsler2} imply item~\ref{item:Clarkebasedcond} in Remark~\ref{rmk:ClarkeRelaxation}, again with $\alpha_x(s):=a_x s^2$. The function $\theta_x\in \cK_\infty$ in~\eqref{eq:xDynamicStrangeFunct} can be defined as $\theta_x(s):=\frac{a_x}{2\|B\|\overline{\lambda}_x}s$.\hfill$\triangle$
\end{myexample}

Let us now consider the error dynamics~\eqref{eq:eSystem} and characterize ISS from $x$, using a $\cC^1$-Lyapunov function, satisfying the next property. 

\begin{property}\label{lemma:ISSe}
Suppose that there exist $V_e\in\cC^1(\R^n,\R)$, and $\underline\psi_e,\overline\psi_e,\alpha_e,\rho_e\in \cK_\infty$ such that
\begin{subequations}
\begin{align} 
&\underline\psi_e(|e|)\leq V_e(e)\leq \overline\psi_e(|e|), \;\;\forall\; e\in \R^n\label{eq:eDynamic1}\\
&\left|\frac{\partial V_e}{\partial e}(e)\right|\leq \rho_e(|e|),\;\;\forall\; e\in \R^n.\label{eq:eDynamic2}
\end{align}
\end{subequations}
Moreover, for all $x\in\R^n $, for all $z\in \R^n$ and for each $i\in\{1,2\}$,
\begin{equation}\label{eq:eDynamic3}
\inp{\frac{\partial V_e}{\partial e}(e)}{f_i(x)-f_i(z)-\ell_i(h(x)-h(z))+\widetilde{g}(x,z)}\leq -\alpha_e(|e|),
\end{equation}
with $e=x-z$.
Finally there exists $\theta_e\in\cK_\infty$  such that 
\begin{equation}\label{eq:eDynamicBoundFunctions}
\theta_e(s)\leq\frac{\alpha_e(s)}{\rho_e(s)}\;\;\;\forall\;s\in \R_+.
\end{equation}
\end{property}

Based on Property~\ref{lemma:ISSe} we can prove the next result.
\begin{prop}\label{prop:EISSfromX}
Under Property~\ref{lemma:ISSe},  there exist $\widehat\gamma_e\in \cK$ and $\widehat \alpha_e\in \cK_\infty$ such that
  \[
 (|e|\geq \widehat \gamma_e(|x|)\,)\;\Rightarrow\;\inp{\frac{\partial V_e}{\partial e}(e)}{f_e}\leq -\widehat\alpha_e(|e|),
  \]
  for all $f_e\in F_{e}(x,e)$, and thus the system \eqref{eq:eSystem} is ISS w.r.t. $x$.
\end{prop}
\begin{proof}
It is easy to see that the second and third expression in \eqref{eq:ErrorDiscDynamic} can be rewritten respectively as
\[
\begin{aligned}
&f_1(x)-f_1(z)-\ell_1(h(x)-h(z))+\widetilde{g}(x,z)+(f_2(x)-f_1(x))\\
&f_2(x)-f_2(z)-\ell_2(h(x)-h(z))+\widetilde{g}(x,z)+(f_1(x)-f_2(x))
\end{aligned}
\]
and thus we can rewrite $f_{e}:\R^n\times \R^n\to \R^n$ as 
\begin{equation*}
f_{e}(x,z):=\begin{cases}
&\begin{aligned}
&f_1(x)-f_1(z)-\ell_1(h(x)-h(z))\\&+\widetilde{g}(x,z)+\mathbb{\cI}_{-}(q(x))\widetilde f(x)\hskip1.1cm\text{if }  q(z)\geq 0,
\end{aligned}\\\\
&\begin{aligned}
&f_2(x)-f_2(z)-\ell_2(h(x)-h(z))\\&+\widetilde{g}(x,z)-\mathbb{\cI}_{+}(q(x))\widetilde f(x)\hskip1.1cm\text{if }  q(z)\leq 0.
\end{aligned}
\end{cases}
\end{equation*}
where we defined $\widetilde f(x):=f_2(x)-f_1(x)$ and $\mathbb{\cI}_{+},\mathbb{\cI}_{-}$ are the indicator functions of the positive and negative real numbers respectively. We note that, by Assumption \ref{ass:mainassumpCONTr},
\[
\max\{|\mathbb{\cI}_{-}(q(x))\widetilde f(x)|,|\mathbb{\cI}_{+}(q(x))\widetilde f(x)|\}\leq \beta_f(|x|)
\]
for all $x\in \R^n$. Thus, we can now apply the same reasoning used in proof of Proposition~\ref{prop:xISSformE}, concluding that
\[
(\,|e|\geq \widehat \gamma_e(|x|)\,)\;\Rightarrow\;\inp{\frac{\partial V_e}{\partial e}(e)}{f_e}\leq -\widehat\alpha_e(|e|), \;\forall f_e\in F_e(x,e),
\]
where $\widehat\alpha_e:=(1-\varepsilon)\alpha_e$ and $\widehat\gamma_e:=(\varepsilon\theta_e)^{-1}\circ\beta_f$, for some $0<\varepsilon<1$. Note that condition~\eqref{eq:eDynamicBoundFunctions} ensures that $\widehat \gamma_e\in \cK$. The ISS property follows again from Theorem~\ref{teo:mainres}.
\end{proof}
To clarify our construction, the idea behind Property~\ref{lemma:ISSe} and Proposition~\ref{prop:EISSfromX} is to search for a common $\cC^1$ Lyapunov function for the two vector fields $f_i(x)-f_i(z)-\ell_i(h(x)-h(z))+\widetilde{g}(x,z)$, $i\in \{1,2\}$. If $x$ and the estimated state $z$ are not in the same region $X_i$, then the $(x-z)$-dynamics is perturbed by a factor $\pm(f_1(x)-f_2(x))$, which is treated as an external disturbance. The injection gains $\ell_i$ induce ISS with respect to these disturbances.
\setcounter{myexample}{2}
\begin{myexample}[Continued]\label{rmk:Linear3}
For the switched linear case presented in Example~\ref{rmk:Linear}, Property~\ref{lemma:ISSe} can be ensured using a quadratic function $V_e(e):=e^\top P_e e$, with $P_e\succ0$, by finding matrices $L_1,L_2\in \R^{n\times p}$ and $a_e>0$ such that
\begin{equation}
P_e(A_i-L_i C)+(A_i-L_i C)^\top P_e+a_e I\prec0,\label{eq:LyapunovV_e}
\end{equation}
for each $i\in \{1,2\}$. Indeed, defining 
\[
\begin{aligned}
&\underline{\lambda_e}:=\lambda_{\text{min}}(P_e),\;\;\overline{\lambda_e}:=\lambda_{\text{max}}(P_e),\;\;\underline{\psi_e}(s):=\underline{\lambda_e}s^2,\;\;\\&\overline{\psi_e}(s):=\overline{\lambda_e}s^2,\;\;\rho_e(s):=2\overline{\lambda_e}s,\;\;\;\widehat \alpha_e(s):=a_e s^2,
\end{aligned}
\]
 equations \eqref{eq:eDynamic1},~\eqref{eq:eDynamic2} and~\eqref{eq:eDynamic3} are satisfied, and the function $\theta_e\in \cK_\infty$ in~\eqref{eq:eDynamicBoundFunctions} is defined as
$\theta_e(s)=\frac{a_e}{2\overline{\lambda_e}} s$.\hfill$\triangle$
\end{myexample}

We are finally ready to state our stability conditions, based on Theorem~\ref{thm:SmallGainThm}, for the interconnection in~\eqref{eq:coupledRegularization}.
\begin{cor}
Assume that Properties~\ref{lemm:ISSx} and~\ref{lemma:ISSe} hold, and define $\eta_1:=\overline\psi_x \circ\widehat\gamma_e \circ  \underline\psi_e^{-1}
$ and
$
\eta_2:= \overline\psi_e\circ\widehat\gamma_x \circ\underline\psi_x^{-1}$.
If
\begin{equation}\label{eq:SmallGAinconditionExample}
\eta_1\circ\eta_2(s)<s, \;\;\forall s>0,
\end{equation}
then system~\eqref{eq:coupledRegularization} is globally asymptotically stable.
\end{cor}
\begin{proof}
Since by Propositions~\ref{prop:xISSformE} and~\ref{prop:EISSfromX} we can construct nonpathological ISS-Lyapunov functions  $V_x$ and $V_e$ as in Assumption~\ref{ass:InterLyapFuncSub}, it remains to check that condition \eqref{eq:SmallGAinconditionExample} implies the small gain condition~\eqref{eq:SamllGainEquation} in Theorem~\ref{thm:SmallGainThm}. First we note that, by~\eqref{eq:xDynamicContinuity} and~\eqref{eq:eDynamic1},
\[
\begin{aligned}
V_x(x)\geq \eta_1(V_e(e))&=\overline{\psi}_x\circ \widehat\gamma_e\circ \underline{\psi_e}^{-1}(V_e(e))\\&\Leftrightarrow\;\;\overline{\psi_x}^{-1}(V(x))\geq \widehat\gamma_e\circ\underline{\psi_e}^{-1}(V_e(e))\\
&\Rightarrow |x|\geq \widehat\gamma_e(|e|),
\end{aligned}
\]
since  $V_x(x)\leq \overline{\psi_x}(|x|)$, and $V_e(e)\geq \underline{\psi_e}(|e|)$. By Proposition~\ref{prop:xISSformE}, this implies that
\[
(V_x(x)\geq \eta_1(V_e(e))\,)\Rightarrow \max \dot{\overline{V}}_{x,F_x}(x,e)\leq - \widehat\alpha_x(|x|).
\]
 Following the same path for $\eta_1$, we obtain the implication
 \[
(V_e(e)\geq \eta_2(V_x(x))\,)\Rightarrow \max \dot{\overline{V}}_{e,F_e}(x,e)\leq - \widehat\alpha_e(|e|),
 \]
 proving that~\eqref{eq:SmallGAinconditionExample} implies~\eqref{eq:SamllGainEquation}.
\end{proof}
\begin{oss}
Considering again the switched system presented in Example~\ref{rmk:Linear}, we can check the small-gain condition \eqref{eq:SmallGAinconditionExample} as follows.
Recalling the definitions of $\widehat\gamma_x$ and $\widehat \gamma_e$ in the proofs of Propositions~\ref{prop:xISSformE} and~\ref{prop:EISSfromX} we can write
\[
\widehat\gamma_x(s)=\frac{2\|B\|\|K\|\overline{\lambda_x}}{\varepsilon a_x}\,s\;\;\text{and }\;\widehat\gamma_e(s)=\frac{2\|A_1-A_2\|\overline{\lambda_e}}{\varepsilon a_e}\,s.
\]
Thus, by arbitrariness of $0<\varepsilon<1$, condition $\eqref{eq:SmallGAinconditionExample}$ holds if 
\[
\frac{16\|B\|^2\|K\|^2\|A_1-A_2\|^2\overline{\lambda_x}^3\overline{\lambda_e}^3}{\underline{\lambda_x}\underline{\lambda_e}a_x^2a_e^2}<1.\tag*{$\triangle$}
\]
\end{oss}

\section{Conclusions}

We focused on ISS of differential inclusions using locally Lipschitz Lyapunov functions.
We provided sufficient conditions based on the notion of Lie derivative of the candidate Lyapunov function, which generalize previous results relying on the study of the Clarke derivative. We applied our results to state-dependent switched systems and proposed a new formulation of the well-known small gain theorem in the context of interconnected differential inclusions.  We finally studied the design of an observer-based output feedback controller for a bimodal switched system. As possible further research, we may investigate convex LMI-based algorithms, based on using Lipschitz non quadratic functions and Lie derivative.

\appendix 
\section{Properties of piecewise $\cC^1$ functions}\label{sec:Appendix}
In this Appendix we prove the two items of Proposition~\ref{lemma:propertiesofppf}, first characterizing Clarke generalized gradient, and then showing that piecewise $\cC^1$ functions are nonpathological.
\begin{lemma}\label{prop:PropoAppendix}
Consider $\cX=\{X_i,\cO_i\}_{i\in \cI}$, a proper partition of $\R^n$. If  $V\in \mathscr{P}(\cX)$ then $V$ is locally Lipschitz and 
\begin{equation}\label{eq:NCLgengrad}
\partial V(x)=\co\left\{ \nabla V_\ell(x) \;\vert\; \ell \in \cI_{\mathcal{X}}(x) \right\}.
\end{equation}
\end{lemma}
\begin{proof}
For the proof that $V$ is locally Lipschitz we refer to \cite[Proposition 4.1.2]{Sch12}.
Given any $x\in \R^n$, define the sets 
\begin{equation*}
\begin{aligned}
S_1(x)&:=\{ \nabla V_\ell(x) \, \vert \, \ell\in \cI_{\mathcal{X}}(x)\},\\
S_2(x)&:=\left\{\lim_{k \to \infty} \nabla V(x_k) \, \big\vert \, x_k \to x, \;x_k \notin \cN_V \right\}.
\end{aligned}
\end{equation*}
We prove below that $S_1(x)=S_2(x)$; then \eqref{eq:NCLgengrad} follows from \eqref{eq:limClark}. 
If $x \in \inn(X_\ell)$ for some $\ell\in \cI$, we have already noted that $\partial V(x)= \{\nabla V_\ell(x)\}$ and thus $S_1(x)=S_2(x)$  trivially holds. Let us suppose $x\notin \bigcup_{\ell=1}^M \inn(X_\ell)$ which implies $x\in \partial X$.

{$S_1(x) \subseteq S_2(x)$:} Consider any $v=\nabla V_\ell(x)$ for some $\ell\in \cI_{\cX}(x)$. By Definition~\ref{def:Ncov}, there exists a sequence $x_k\in \inn(X_\ell)$ such that $x_k\to x$ and thus $\nabla V(x_k)=\nabla V_\ell(x_k)$. By continuity of $\nabla V_\ell$ we have
$
\lim_{k\to \infty}\nabla V(x_k)=\lim_{k\to \infty}\nabla V_\ell(x_k)=\nabla V_\ell(x),
$ therefore $v\in S_2(x)$.

{$S_2(x) \subseteq S_1(x)$:} Consider any $v\in \R^n$ such that there exists a sequence $x_k \to x$ of points where $V$ is differentiable, such that the sequence $\nabla V(x_k)$ converges to $v$. By Definition~\ref{def:Ncov}, there exists a neighborhood $\cO$ of $x$ such that $\cO\subseteq \bigcup_{\ell\in \cI_\cX(x)}X_\ell$, and thus for each $k$ (large enough) there exists an index $\ell_k\in \cI_\cX(x)$ satisfying $x_k\in X_{\ell_k}$. By finiteness of $\cI_\cX(x)$ we can extract a subsequence of $x_k$ (without relabeling) and an index $\ell\in \cI_\cX(x)$ such that $x_k\in X_\ell$ for each $k\in \N$ and $\nabla V(x_k)\to v$. Now for each $k\in\N$ we can take a sequence $x_{k,l}\to x_k$ satisfying $x_{k,l}\in \inn(X_\ell)$. Thus 
$
\nabla V(x_k)=\lim_{l\to \infty}\nabla V(x_{k,l})=\lim_{l\to \infty}\nabla V_\ell(x_{k,l})=\nabla V_\ell(x_k).
$
Summarizing, we have
\[
v=\lim_{k\to\infty} \nabla V(x_k)=\lim_{k\to\infty} \nabla V_\ell(x_k)=\nabla V_\ell(x).
\]

Since we proved $S_1(x)=S_2(x)$, then $\eqref{eq:NCLgengrad}$ follows from~\eqref{eq:limClark}.
\end{proof}
Having shown~\eqref{eq:NCLgengrad}, equation~\eqref{eq:PPFLieDerivative} in Proposition~\ref{lemma:propertiesofppf} directly follows. The next statement completes the proof of Proposition~\ref{lemma:propertiesofppf}.

\begin{lemma}
Consider $\cX=\{X_i,\cO_i\}_{i\in \cI}$, a \emph{proper partition} of $\R^n$. If  $V\in \mathscr{P}(\cX)$, then $V$ is nonpathological. 
\end{lemma}
\begin{proof}
Recalling Definition~\ref{def:nonpat} we must show that, given any $\varphi \in AC(\R_+, \R^n)$, $\partial V(\varphi(t))$ is  a subset of an affine subspace orthogonal to $\dot \varphi(t)$, for almost all $t\in \R_+$, namely that $\exists a_t \in \R$ such that
\begin{equation}\label{eq:AppendixAffineSUbspace}
\inp{v}{\dot \varphi(t)}=a_t, \hskip0.5cm \forall v \in \partial V(\varphi(t)).
\end{equation}
Since $\varphi:\R_+\to\R^n$ is absolutely continuous and $V:\R^n\to \R$ is locally Lipschitz we have  that $\varphi$ and $V(\varphi(t))$ are differentiable almost everywhere, 
i.e. there exists a set of measure zero $\cN\subseteq\R_+$ such that 
$\dot{\varphi}(t)$ and $\frac{d}{dt}V(\varphi(\cdot))(t)$ both exist for every $t\in \R_+\setminus\cN$. 
Using~\eqref{eq:NCLgengrad} in Lemma~\ref{prop:PropoAppendix}, to ensure~\eqref{eq:AppendixAffineSUbspace} it is enough to show that, for almost all $t\in \R_+\setminus \cN$, there exists $a_t \in \R$ such that 
\begin{equation}\label{eq:AppendixEquation2}
\inp{\nabla V_\ell(\varphi(t))}{\dot \varphi(t)}=a_t, \;\;\forall\,\ell \in \cI_{\mathcal{X}}(\varphi(t)).
\end{equation}
Fix any $t\in \R_+\setminus \cN$. Either~\eqref{eq:AppendixEquation2} holds for that $t$, or there exist $\ell_1,\ell_2 \in \cI_{\mathcal{X}}(\varphi(t))$ such that $\ell_1\neq\ell_2$ and 
\[
\inp{\nabla V_{\ell_1}(\varphi(t))}{\dot \varphi(t)}\neq \inp{\nabla V_{\ell_2}(\varphi(t))}{\dot \varphi(t)}.
\]
In this second case we have
\[
\begin{aligned}
\frac{d}{dt}&\left(V_{\ell_1}(\varphi(t))-V_{\ell_2}(\varphi(t))\right)\\ =&\inp{\nabla V_{\ell_1}(\varphi(t))}{\dot \varphi(t)}-\inp{\nabla V_{\ell_2}(\varphi(t))}{\dot \varphi(t)}\neq 0.
\end{aligned}
\]
Thus, by continuity, there exists $\varepsilon >0$  small enough such that $V_{\ell_1}(\varphi(\widetilde t))\neq V_{\ell_2}(\varphi(\widetilde t))$ , for all $\widetilde t \in (t-\varepsilon,t+\varepsilon) \setminus \{t\}$, which implies that either $\ell_1\notin\cI_{\mathcal{X}}(\varphi(\widetilde t))$ or $\ell_2\notin\cI_{\mathcal{X}}(\varphi(\widetilde t))$ (or both), for all such $\widetilde t$, since, by Definition \ref{defn:piecewise},
\[
\ell_1,\ell_2\in \cI_{\cX}(x)\;\;\Rightarrow\;\;V(x)=V_{\ell_1}(x)=V_{\ell_2}(x),
\]
for any $x\in \R^n$. Iterating the argument, this shows that for any point $t$ where two or more scalar products $\inp{\nabla V_{\ell_j}(\varphi(t))}{ \dot \varphi(t)}$ \virgolette{disagree} in~\eqref{eq:AppendixEquation2}, $t$ is isolated. We conclude by recalling that a set of isolated point is countable \cite[Page 180]{hrbacek1999introduction} and thus has measure zero, as to be proven.
\end{proof}



\bibliographystyle{plain}
\bibliography{biblio1}

\end{document}